%% file: homquad_v20.tex
\documentclass[10pt,epsfig]{amsart}
\usepackage{graphicx,color,epsf}
\usepackage [cmtip,arrow]{xy}
\usepackage {pb-diagram,pb-xy} 

\newtheorem{theorem}{Theorem}[section]
\newtheorem{lemma}[theorem]{Lemma}

\newtheorem{proposition}[theorem]{Proposition}

\theoremstyle{definition}
\newtheorem{definition}[theorem]{Definition}
\newtheorem{example}[theorem]{Example}

\theoremstyle{remark}
\newtheorem{remark}[theorem]{Remark}

\numberwithin{equation}{section}

\newcommand {\hide}[1]{}

\newtheorem{notation}{Notation}

\newcommand {\junk}[1]{}

\newcommand {\R} {\mbox{\rm R}}
\newcommand {\s}        {\mbox{\rm sign}}

\newcommand {\A}     {{\mathcal A}}

\newcommand {\Real}[1]   {\mbox{${\Bbb R}^{#1}$}}
\newcommand {\Sphere}{\mbox{${\bf S}$}}     
\newcommand {\Suspension}{{\mathbf S}}     

 \newcommand {\re}         {\Real{}}

\newcommand {\Z}  {{\mathbb Z}}

\newcommand {\ZZ} {{\rm Z}}
\newcommand {\RR} {{\mathcal R}}

\newcommand {\la}   {{\langle}}
\newcommand {\ra}   {{\rangle}}
\newcommand {\eps} {{\varepsilon}}
\newcommand {\E} {{\rm Ext}}

\newcommand {\PP}     {{\Bbb P}} 

\newcommand {\spanof} {{\rm span}}

\newcommand{\x}{\mathbf{x}}

\newcommand{\y}{\mathbf{y}}

\newcommand{\z}{\mathbf{z}}

\newcommand{\hocolimit}{{\rm hocolim}}
\newcommand{\colimit}{{\rm colim}}
\newcommand{\HH}{{\rm H}}
\newcommand{\bigcupdot}
{\mathop{\makebox[0pt]{\hskip 1.4em $\boldsymbol\cdot$}\bigcup}}

\begin{document}
\title[Bounding the number of homotopy types]
{Bounding the number of stable homotopy types 
of a parametrized family of semi-algebraic sets defined by quadratic 
inequalities}

\author{Saugata Basu}
\author{Michael Kettner}
\address{School of Mathematics,
Georgia Institute of Technology, Atlanta, GA 30332, U.S.A.}

\email{\{saugata,mkettner\}@math.gatech.edu}





\begin{abstract}
We prove a nearly optimal bound on the number of stable 
homotopy types occurring 
in a $k$-parameter semi-algebraic 
family of sets in $\R^\ell$, each
defined in terms of $m$ quadratic inequalities. Our bound is exponential
in $k$ and $m$, but polynomial in $\ell$. 
More precisely, we prove the following.
Let $\R$ be a real closed field and let
\[
{\mathcal P} = \{P_1,\ldots,P_m\} \subset \R[Y_1,\ldots,Y_\ell,X_1,\ldots,X_k],
\]
with
${\rm deg}_Y(P_i) \leq 2, {\rm deg}_X(P_i) \leq d, 1 \leq i \leq m$.
Let $S \subset \R^{\ell+k}$ be a semi-algebraic set, defined by a Boolean 
formula without negations, whose atoms are of the form,
$P \geq 0, P\leq 0, P \in {\mathcal P}$.
Let $\pi: \R^{\ell+k} \rightarrow \R^k$ be the projection on the
last $k$ co-ordinates.
Then, the number of stable
homotopy types amongst the fibers 
$S_{\x} = \pi^{-1}(\x) \cap S$ 
is bounded by 
$
\displaystyle{
(2^m\ell k d)^{O(mk)}.
}
$
\end{abstract}
\maketitle
\section{Introduction}
\label{sec:intro}
Let $S \subset \R^{\ell+k}$ be a semi-algebraic set over a real closed field
$\R$. 
Let ${\pi:\R^{\ell + k} \to \R^k}$ be the projection map
on the last $k$ co-ordinates, and for any $S \subset \R^{\ell+k}$
we will denote by $\pi_S$ the restriction of $\pi$ to $S$. 
Moreover, when the map $\pi$ is clear from context, 
for any $\x \in \R^k$ we will denote by $S_{\x}$ the fiber
$\pi_S^{-1}(\x)$.

A fundamental theorem in semi-algebraic geometry states,

\begin{theorem}(Hardt's triviality theorem \cite{Hardt})
\label{the:hardt}
There exists a semei-algebraic partition of $\R^k$,
$\{T_i\}_{i \in I}$, such that the map
$\pi_{S}$ is definably trivial over each $T_i$.
\end{theorem}

Theorem \ref{the:hardt} implies
that for each $i \in I$ and any point
$\y \in T_i$, the pre-image $\pi^{-1}(T_i) \cap S$ 
is semi-algebraically homeomorphic to $S_{\y} \times T_i$ 
by a fiber preserving homeomorphism.
In particular, for each $i \in I$, all fibers $S_{\y}$, 
$\y \in T_i$, 
are semi-algebraically homeomorphic.

Hardt's theorem is a corollary of
the existence of {\em cylindrical algebraic  decompositions}
(see \cite{BPR03}), 
which implies
a double exponential (in $k$ and $\ell$) 
upper bound on the cardinality of $I$ and
hence on the number of homeomorphism types of the fibers of the map $\pi_S$.
No better bounds than the double exponential bound are known, 
even though
it seems reasonable to conjecture a single exponential upper bound
on the number of homeomorphism types of the fibers of the map $\pi_S$.

In \cite{BV06}, the weaker problem of bounding the number of
distinct {\em homotopy types}, occurring amongst the set of all
fibers of $\pi_S$ was considered,
and a single exponential  upper bound was proved on the number of
homotopy types of such fibers.

Before stating this result more precisely we need to introduce a 
few notation. Let $\R$ be a real closed field,
${\mathcal P} \subset \R[Y_1,\ldots,Y_\ell,X_1,\ldots,X_k]$, and
let $\phi$ be a Boolean formula with atoms of the form
$P=0$, $P > 0$, or $P< 0$, where $P \in {\mathcal P}$.
We call $\phi$ a ${\mathcal P}$-formula, and the semi-algebraic set
$S \subset \R^{\ell+k}$ defined by $\phi$, a ${\mathcal P}$-semi-algebraic set.

If the Boolean formula $\phi$ contains no negations, and
its atoms are of the form
$P= 0$, $P \geq 0$, or $P \leq  0$, with $P \in {\mathcal P}$,
then we call $\phi$ a ${\mathcal P}$-closed formula, and the semi-algebraic set
$S \subset \R^{\ell+k}$ defined by ${\phi}$,
a ${\mathcal P}$-closed semi-algebraic set.

The following theorem appears in \cite{BV06}.

\begin{theorem}\cite{BV06}
\label{the:mainBV}
Let ${\mathcal P} \subset \R[Y_1,\ldots,Y_\ell,X_1,\ldots,X_k]$,
with $\deg(P) \leq d$ for each $P \in {\mathcal P}$ and cardinality
$\#{\mathcal P} = m$.
Then, there exists a finite set~$A \subset \R^k$,
with 
\[
\# A \leq (2^\ell mkd)^{O(k\ell)},
\]
such that for every $\x \in \R^k$ there exists $\z \in A$ such that
for every ${\mathcal P}$-semi-algebraic set~$S \subset \R^{\ell+k}$, the set~$S_{\x}$
is semi-algebraically homotopy equivalent to $S_{\z}$.
In particular, for any fixed ${\mathcal P}$-semi-algebraic set~$S$,
the number of different homotopy types of fibers~$S_{\x}$
for various $\x \in \pi(S)$ is also bounded by
\[
(2^{\ell} mkd)^{O(k\ell)}.
\]
\end{theorem}

A result similar to Theorem \ref{the:mainBV}
has been proved for semi-Pfaffian sets as well in \cite{BV06},
and has been extended to arbitrary o-minimal structures in
\cite{Basu07b}.
The bounds on the number of homotopy types proved in 
\cite{BV06,Basu07b} are all exponential in $\ell$ as well as $k$.
The following example, which appears in \cite{BV06}, shows that in this generality
the single exponential dependence on $\ell$ 
is unavoidable.

\begin{example}
\label{eg:exp}
Let
$P \in \R[Y_1,\ldots,Y_\ell] \hookrightarrow \R[Y_1,\ldots,Y_\ell,X]$
be the polynomial defined by
\[
P = \sum_{i=1}^{\ell} \prod_{j=0}^{d-1}(Y_i - j)^2.
\]
The algebraic set defined by $P=0$ in
$\R^{\ell+1}$
with co-ordinates $Y_1,\ldots,Y_\ell,X$,
consists of $d^\ell$ lines all parallel to the $X$ axis.
Consider now the semi-algebraic set~$S \subset \R^{\ell+1}$ defined by
\[
\displaylines{
(P = 0)\; \wedge \;
(0 \le X \le Y_1+ dY_2 + d^2 Y_3 + \cdots + d^{\ell-1}Y_\ell).
}
\]
It is easy to verify that, if $\pi:\> \R^{\ell+1} \to \R$ is the
projection map on the $X$ coordinate, then
the fibers~$S_{\x}$, 
for $\x \in \{0,1, 2, \ldots ,d^\ell-1 \} \subset \R$
are 0-dimensional and of different cardinality, 
and hence have different homotopy types.
\end{example}
%
\subsection{Semi-algebraic sets defined by quadratic inequalities}
%
One particularly interesting class of semi-algebraic sets
is the class of semi-algebraic sets defined
by quadratic inequalities. This class of sets
has been investigated from an algorithmic standpoint
\cite{Barvinok93,GP,Basu05a,Basu07a,BZ}, 
as well as from the point of view topological complexity,  
\cite{Agrachev,Barvinok97,BK}. 
 
Semi-algebraic sets defined by quadratic inequalities 
are distinguished
from arbitrary semi-algebraic sets by the fact that, 
if the number of  inequalities is fixed, then the
sum of their Betti numbers is bounded polynomially in the dimension.
The following bound was proved by Barvinok \cite{Barvinok97}.

\begin{theorem}
\label{the:barvinok}
Let $S \subset \R^\ell$ be a semi-algebraic set defined by the inequalities,
$P_1 \geq 0,\ldots,P_m \geq 0$, $\deg(P_i) \leq 2, 1 \leq i \leq m$.
Then,
$\displaystyle{
\sum_{i=0}^\ell b_i(S) \leq (m \ell)^{O(m)},
}
$
where $b_i(S)$ denotes the $i$-th Betti number of $S$.
\end{theorem}

An extension of Barvinok's bound to arbitrary ${\mathcal P}$-closed
(not just basic closed)
semi-algebraic sets defined in terms of quadratic inequalities has been
done recently in \cite{BP'R07}.

Now suppose that we have a parametrized family of sets, each defined 
in terms of $m$ quadratic inequalities.
More precisely, let 
\[
{\mathcal P} = \{P_1,\ldots,P_m\} \subset \R[Y_1,\ldots,Y_\ell,X_1,\ldots,X_k],
\]
with
$
{\rm deg}_Y(P_i) \leq 2, {\rm deg}_X(P_i) \leq d, 1 \leq i \leq m
$
($X_1,\ldots,X_k$ are the {\em parameters}),
and let $S \subset \R^{\ell+k}$ be a ${\mathcal P}$-closed semi-algebraic set.
Let $\pi: \R^{\ell+k} \rightarrow \R^k$ denote the projection on the
last $k$ co-ordinates.
Then, for each $\x \in \R^k$ the semi-algebraic set~$S_{\x}$ 
is defined by a Boolean formula involving at most
$m$ quadratic  polynomials in $Y_1,\ldots,Y_\ell$.

Bounding the number of topological types amongst the fibers, $S_{\x}, \x \in R^k$, 
is an interesting special case of the more general problem
mentioned in the last section.
In view of the topological simplicity
of semi-algebraic sets defined by few quadratic inequalities 
as opposed to general semi-algebraic sets
(cf. Theorem~\ref{the:barvinok}), 
one might expect a much tighter bound on the number
of topological types compared to the general case.
However one should be cautious, since a tight  bound on the Betti numbers 
of a class of semi-algebraic sets
does not automatically  imply a similar bound on 
the number of topological or even homotopy types occurring in that class.  

In this paper we consider the problem of bounding the number of 
\textit{stable homotopy types} 
(see Definition \ref{def:S-equivalence} below)
of fibers~$S_{\x}$, 
where $\pi$ and $S$ are as defined above. 
We prove a bound which for each fixed $m$, 
is polynomial in $\ell$ (the dimension of the fibers). In some 
special cases our bound can be extended to the number of homotopy types 
(see Theorem~\ref{the:union}). 

Our result can be seen as a follow-up to the recent work on 
bounding the number of homotopy types of fibers of general semi-algebraic
maps studied in \cite{BV06}. However, the bound in \cite{BV06} applied to the 
special case of sets defined by quadratic inequalities
would yield a bound exponential in both $k$ and $\ell$, 
as shown by Example~\ref{eg:exp}, 
where the semi-algebraic set~$S$ is defined in terms of three polynomials.

\begin{remark}
Note that the notions of homeomorphism type, homotopy
type and stable homotopy type are each strictly weaker than the previous
one, since two semi-algebraic sets might be stable homotopy equivalent,
without being homotopy equivalent 
(see \cite{Spanier}, p.~462), and also homotopy
equivalent without being homeomorphic. 
However, two closed and bounded semi-algebraic sets which are
stable homotopy equivalent have isomorphic homology groups.
\end{remark}

\subsection{Prior and Related Work}
%
Since sets defined by quadratic equalities and inequalities
are the simplest class of topologically non-trivial semi-algebraic
sets, the problem of classifying such sets topologically has 
attracted the attention of many researchers.

Motivated by problems related to stability of maps,
Wall \cite{Wall}  considered the special case of real 
algebraic sets defined by two simultaneously
diagonalizable quadratic forms in $\ell$ variables. 
He obtained a full topological classification
of such varieties making use of Gale diagrams (from the theory of 
convex polytopes).
In our notation,
letting 
$$
\displaylines{
Q_1=\sum_{i=1}^\ell X_iY_i^2, \cr
Q_2=\sum_{i=1}^\ell X_{i+\ell} Y_i^2,
}
$$
and
\[
S=\{(\y,\x)\in\re^{3\ell} \mid \quad\parallel\y\parallel=1,\quad Q_1(\y,\x)=Q_2(\y,\x)=0\},
\]
Wall obtains as a consequence  of his classification theorem,
that the number of different topological types of 
fibers~$S_{\x}$ 
is  bounded by $2^{\ell-1}$.
Notice that in this case
the number of parameters ($X_1,\ldots,X_{2\ell}$), as well
as the number of variables ($Y_1,\ldots,Y_\ell$), 
are both $O(\ell)$.
Similar results were also obtained by L\'opez \cite{Lopez}
using different techniques.
Much more recently Briand \cite{Briand07} 
has obtained explicit characterization
of the isotopy classes of real varieties defined by two general 
conics in $\re \PP^2$
in terms of the coefficients of the polynomials. His method 
also gives a decision
algorithm for testing whether two such given varieties are isotopic.

In another direction
Agrachev \cite{Agrachev} studied the topology of semi-algebraic sets
defined by quadratic inequalities, and he defined a certain spectral sequence
converging to the homology groups of such sets. A parametrized version
of Agrachev's construction is in
fact a starting point of our proof of the main theorem in this paper. 

%
\section{Main Result}
The main result of this paper is the following theorem.
\begin{theorem}
\label{the:main}
Let $\R$ be a real closed field and let
\[
{\mathcal P} = \{P_1,\ldots,P_m\} \subset \R[Y_1,\ldots,Y_\ell,X_1,\ldots,X_k],
\]
with
${\rm deg}_Y(P_i) \leq 2, {\rm deg}_X(P_i) \leq d, 1 \leq i \leq m$.
Let $\pi: \R^{\ell+k} \rightarrow \R^k$ be the projection on the
last $k$ co-ordinates. 
Then,
for any ${\mathcal P}$-closed semi-algebraic set~$S \subset \R^{\ell+k}$,
the number of stable homotopy types amongst the fibers, $S_{\x}$, is bounded by 
$\displaystyle{
(2^m\ell k d)^{O(mk)}.
}
$
\end{theorem}
\begin{remark}
Note that the bound in Theorem \ref{the:main} (unlike that in
Theorem \ref{the:mainBV}) is polynomial in $\ell$ for fixed $m$ and
$k$. 
The exponential dependence on $m$ is unavoidable, as can be seen from
a slight modification of Example \ref{eg:exp} above.
Consider the semi-algebraic set~$S \subset \R^{\ell+1}$ defined by
\[
\displaylines{
Y_i(Y_i-1) = 0, \; 1 \leq i \leq  m \leq \ell, \cr
0 \leq X \leq Y_1 +2\cdot Y_2 + \ldots + 2^{m-1}\cdot Y_m.
}
\]

Let $\pi: \R^{\ell+1} \rightarrow \R$ be the projection on the
$X$-coordinate. Then, 
the sets~$S_{\x}$, ${\x \in \{0,1\ldots,2^{m-1}\}}$, 
have different number of connected
components, and hence have distinct (stable) 
homotopy types.
\end{remark}

\begin{remark}
Note that the technique used to prove Theorem \ref{the:mainBV}
in \cite{BV06} does not directly produce better bounds in the quadratic case,
and hence we need a new approach to prove a substantially better bound
in this case. 
However, due to technical reasons, we only obtain 
a bound on the number of stable homotopy types, rather than homotopy types.
\end{remark}

\section{Mathematical Preliminaries}
%
We first need to fix some notation and a few preliminary results
needed later in the proof.
%
\subsection{Some Notation}\label{sub:notations}
Let $\R$ be  a real closed field. For an element $a \in \R$ introduce
\[
\s(a) = \Biggl\{
\begin{tabular}{ccc}
0 & \mbox{ if }  $a=0$,\\
1 & \mbox{ if } $a> 0$,\\
$-1$& \mbox{ if } $a< 0$.
\end{tabular}
\]
If ${\mathcal P} \subset {\R} [X_1, \ldots , X_k]$ is finite, we write the
{\em set of zeros} of ${\mathcal P}$ in ${\R}^k$ as
$$
\ZZ({\mathcal P})=\Bigl\{ \x \in {\R}^k\mid\bigwedge_{P\in{\mathcal P}}P(\x)= 0 \Bigr\}.
$$
A  {\em sign condition} $\sigma$ on
${\mathcal P}$ is an element of $\{0,1,- 1\}^{\mathcal P}$.
The {\em realization of the sign condition
$\sigma$} is the basic semi-algebraic set
$$
\RR(\sigma) = \Bigl\{ \x \in {\R}^k\;\mid\;
\bigwedge_{P\in{\mathcal P}} \s({P}(\x))=\sigma(P) \Bigr\}.
$$
A sign condition $\sigma$ is {\em realizable} if $\RR(\sigma) \neq \emptyset$.
We denote by ${\rm Sign}({\mathcal P})$ the set of realizable sign conditions
on ${\mathcal P}.$
For $\sigma \in {\rm Sign}({\mathcal P})$ we define the {\em level of} $\sigma$
as the cardinality $\#\{P \in {\mathcal P}| \sigma(P) = 0 \}$.
For each level $p$, $0 \leq p \leq \# {\mathcal P}$, we denote by
${\rm Sign}_{p}({\mathcal P})$
the subset of ${\rm Sign}({\mathcal P})$ of elements of level $p$.
Moreover, for a sign condition $\sigma$ let
\[
{\mathcal Z}(\sigma) = \Bigl\{ \x \in {\R}^k\;\mid\;
\bigwedge_{P\in{\mathcal P},\ \sigma (P)=0} P(\x)=0 \Bigr\}.
\]
%

\subsection{Use of Infinitesimals}
%
Later in the paper,
we will extend the ground field $\R$ by infinitesimal
elements.
We denote by $\R\langle \zeta\rangle$  the real closed field of algebraic
Puiseux series in $\zeta$ with coefficients in $\R$ (see \cite{BPR03} for
more details). 
The sign of a Puiseux series in $\R\langle \zeta\rangle$
agrees with the sign of the coefficient
of the lowest degree term in
$\zeta$. 
This induces a unique order on $\R\langle \zeta\rangle$ which
makes $\zeta$
infinitesimal: $\zeta$ is positive and smaller than
any positive element of $\R$.
When $a \in \R\la \zeta \ra$ is bounded 
from above and below by some elements of $\R$,
$\lim_\zeta(a)$ is the constant term of $a$, obtained by
substituting 0 for $\zeta$ in $a$.
Given a semi-algebraic set
$S$ in ${\R}^k$, the {\em extension}
of $S$ to $\R'$, denoted $\E(S,\R'),$ is
the semi-algebraic subset of ${ \R'}^k$ defined by the same
quantifier free formula that defines $S$.
The set~$\E(S,\R')$ is well defined (i.e. it only depends on the set
$S$ and not on the quantifier free formula chosen to describe it).
This is an easy consequence of the Tarski-Seidenberg transfer principle 
(see for instance
\cite{BPR03}).

We will also need the following remark about extensions which
is again a consequence of the Tarski-Seidenberg transfer principle.
\begin{remark}
\label{rem:transfer}
Let $S,T$ be two closed and bounded semi-algebraic subsets of $\R^k$, and
let $R'$ be a real closed extension of $\R$. Then, $S$ and $T$ are
semi-algebraically homotopy equivalent if and only if
$\E(S,\R')$ and $\E(T,\R')$  are semi-algebraically homotopy equivalent.
\end{remark}

We will need a few results from algebraic topology, which we state here
without proof referring the reader to papers where the proofs appear. 

The following inequalities are consequences of the Mayer-Vietoris exact
sequence.
\subsection{Betti numbers and Mayer-Vietoris Inequalities}
%
We will use the following notation.
\begin{notation}
\label{not:[m]}
For each $m \in \Z_{\geq 0}$ we will denote by $[m]$ the set $\{1,\ldots, m\}$.
\end{notation}

\begin{proposition}[Mayer-Vietoris inequalities]
\label{prop:MV}
Let the subsets 
$W_1, \ldots , W_r \subset \R^n$ be all open or all closed.
Then, for each $i \geq 0$ we have,
\begin{equation}\label{eq:MV1}
{b}_i \left( \bigcup_{1 \le j \le r} W_j \right) \le \sum_{J \subset [r]}
{b}_{i- (\# J) +1} \left( \bigcap_{j \in J} W_j \right)
\end{equation}
and
\begin{equation}\label{eq:MV2}
{b}_i \left( \bigcap_{1 \le j \le r} W_j \right) \le \sum_{J \subset [r]}
{b}_{i + (\# J) -1} \left( \bigcup_{j \in J} W_j \right).
\end{equation}
\end{proposition}
\begin{proof}
See \cite{BPR03}.
\end{proof}
The following proposition gives a bound on the Betti numbers of the
projection $\pi(V)$ of a closed and bounded semi-algebraic set~$V$
in terms of the number and degrees of polynomials defining $V$.

\begin{proposition}\cite{GVZ}
\label{prop:GVZ}
Let $R$ be a real closed field and let 
$\psi:\R^{m + k} \to \R^k$ be the projection map on to last $k$ co-ordinates. 
Let $V \subset {\R}^{m+k}$ be a closed and bounded semi-algebraic set 
defined by a Boolean formula with $s$ distinct polynomials of 
degrees not exceeding $d$. 
Then the $n$-th Betti number of the projection
\[
{\rm b}_n (\psi (V)) \le (nsd)^{O(k+nm)}.
\]
\end{proposition}
\begin{proof}
See \cite{GVZ}.
\end{proof}

\subsection{Stable homotopy equivalence}
\label{subsec:stable}
For any finite 
CW-complex $X$ we will denote by $\Suspension(X)$ the suspension of 
$X$.

Recall from \cite{Spanier-Whitehead}
that for two finite CW-complexes $X$ and $Y$, an element of
\begin{equation}
\label{eqn:defofS-maps}
\{X;Y\}= \varinjlim_i \; [\Suspension^i(X),\Suspension^i(Y)]
\end{equation}
is called an {\em S-map} (or map in the {\em suspension category}).
(When the context is clear we will sometime denote an S-map $f \in \{X;Y\}$ by
$f: X \rightarrow Y$).

\begin{definition}
\label{def:S-equivalence}
An S-map $f \in \{X;Y\}$ is an S-equivalence 
(also called a stable homotopy equivalence) if it admits an
inverse $f^{-1} \in \{Y;X\}$. In this case we say that $X$ and $Y$ are
stable homotopy equivalent.
\end{definition}

If $f \in \{X;Y\}$ is an S-map, then $f$ induces a homomorphism,
\[
f_* : \HH_*(X) \rightarrow \HH_*(Y).
\]

The following theorem characterizes stable homotopy equivalence in terms of
homology.

\begin{theorem}\cite{Spanier}
\label{the:stable}
Let $X$ and $Y$ be two finite CW-complexes. 
Then $X$ and $Y$ are stable homotopy
equivalent if and only if  there exists an S-map
$f \in \{X;Y\}$ 
which induces isomorphisms $f_* : \HH_i(X) \rightarrow \HH_i(Y)$
(see \cite{Dieudonne}, pp. 604).
\end{theorem}

\subsubsection{Spanier-Whitehead duality}
%
In order to compare the complements of closed and 
bounded semi-algebraic sets which are homotopy equivalent,
we will use the duality theory due to Spanier and Whitehead
\cite{Spanier-Whitehead}.
We will need the following facts about Spanier-Whitehead duality
(see \cite{Dieudonne}, pp. 603 for more details).
Let $X \subset \Sphere^n$ be a finite CW-complex. Then there exists
(up to stable homotopy equivalence)
a dual complex, denoted $D_n X \subset \Sphere^n \setminus X$.
The dual complex $D_n X$ is defined only up to S-equivalence.
In particular, any deformation retract of $\Sphere^n \setminus X$ represents
$D_n X$.
Moreover, the functor $D_n$ has the following property.
If $Y \subset \Sphere^n$ is another finite CW-complex,
and  
the S-map represented by $\phi: X \rightarrow Y$ 
is a stable homotopy equivalence, then
there exists a stable homotopy equivalence 
$D_n \phi$.
Moreover, if the map $\phi: X \rightarrow Y$ is an inclusion, then 
the dual S-map $D_n \phi$ is also represented by a corresponding inclusion.

\begin{remark}
\label{rem:spanier-whitehead}
Note that, since Spanier-Whitehead duality theory deals only with
finite polyhedra over $\re$, it extends without difficulty to general real
closed fields using the Tarski-Seidenberg transfer principle.
\end{remark}
%
\subsection{Homotopy colimits}
%
Let ${\mathcal A} = \{A_1,\ldots,A_n\}$,  where each $A_i$ is a sub-complex
of a finite CW-complex.

Let $\Delta_{[n]}$ denote the standard simplex of dimension $n-1$ with
vertices in $[n]$.
For $I \subset [n]$, we denote by $\Delta_I$ 
the $(\#I-1)$-dimensional face of $\Delta_{[n]}$ corresponding
to $I$, and by $A_I$ the CW-complex 
$\displaystyle{\bigcap_{i \in I} A_i}$.

The homotopy colimit, $\hocolimit(\A)$,  is a CW-complex defined as follows.
\begin{definition}
\label{def:hocolimit}
\[ 
\hocolimit(\A) =  \bigcupdot_{I \subset [n]} \Delta_I \times A_I/\sim
\]
where the equivalence relation $\sim$ is defined as follows.

For $I \subset J \subset [n]$, let $s_{I,J}: \Delta_I \hookrightarrow \Delta_J$
denote the inclusion map of the face $\Delta_I$ in $\Delta_J$, and let
$i_{I,J}: A_J \hookrightarrow A_I$ denote the inclusion map of
$A_J$ in $A_I$.

Given $({\mathbf s},\x) \in \Delta_I \times A_I$ and 
$({\mathbf t},\y) \in \Delta_J \times A_J$ with $I \subset J$, 
then $({\mathbf s},\x) \sim 
({\mathbf t},\y)$ if and only if
${\mathbf t} = s_{I,J}({\mathbf s})$ and $\x = i_{I,J}(\y)$.
\end{definition}

We have a obvious map 
\[
f: \hocolimit(\A) \longrightarrow \colimit(\A) = \bigcup_{i \in [n]} A_i
\]
sending $({\mathbf s},\x) \mapsto \x$. 
It is a consequence of the Smale-Vietoris  theorem \cite{Smale} that

\begin{lemma}
\label{lem:hocolimit1}
The map 
\[
f: \hocolimit(\A) \longrightarrow \colimit(\A) = \bigcup_{i \in [n]} A_i
\]
is a homotopy equivalence.
\end{lemma}

Now let $\A = \{A_1,\ldots,A_n\}$ (resp. ${\mathcal B} = \{B_1,\ldots,B_n\}$) 
be a set of sub-complexes of a finite CW-complex.
For each $I \subset [n]$ let 
$f_I \in \{A_I;B_I\}$ be a stable homotopy
equivalence, having the property that for each $I \subset J \subset [n]$,
$f_J = f_I|_{A_J}$. 
Then, we have an induced S-map,
$f \in \{\hocolimit(\A);\hocolimit({\mathcal B})\}$,  and we have that
\begin{lemma}
\label{lem:hocolimit2}
The induced S-map $f\in \{\hocolimit(\A); \hocolimit({\mathcal B})\}$ 
is a stable homotopy equivalence.
\end{lemma}

\begin{proof}
Using the Mayer-Vietoris exact sequence it is easy to see that if the $f_I$'s
induce isomorphisms in homology, so does the map $f$. Now apply Theorem
\ref{the:stable}.
\end{proof}
\section{Proof of Theorem \ref{the:main}}

\subsection{Proof Strategy}
%
The strategy underlying  our proof of Theorem \ref{the:main} is as follows.
We first consider the special case of a
semi-algebraic subset, $A \subset \Sphere^{\ell}$, 
defined by a disjunction of $m$ homogeneous quadratic inequalities restricted
to the unit sphere in $\R^{\ell+1}$. 
We then show that there exists a closed and bounded
semi-algebraic set~$C'$ (see (\ref{eqn:defofC'})
below for the precise definition of the semi-algebraic set~$C'$),
consisting of certain sphere bundles, glued along 
certain sub-sphere bundles,
which is homotopy equivalent to $A$. 
The number of these sphere bundles,
as well descriptions of their bases, 
are bounded polynomially in $\ell$ (for fixed $m$).  

In the presence of parameters $X_1,\ldots,X_k$, 
the set~$A$, as well as $C'$, will depend on the values of the
parameters. However, 
using some basic homotopy properties of bundles, we show that
the homotopy type of the set~$C'$ stays invariant,
under continuous deformation  of the bases  
of the different sphere bundles which constitute $C'$.
These bases also depend on the parameters, $X_1,\ldots,X_k$,
but the
degrees of the polynomials defining them have degrees bounded
by $O(\ell d)$ in $X_1,\ldots,X_k$.
Now, using techniques similar to those used in \cite{BV06}, 
we are able to control the number of 
isotopy types of the bases which  occur, as the parameters
vary over $\R^k$. The bound on the number of isotopy types, also gives a 
bound on the number of possible homotopy types of the set~$C'$ 
and hence of $A$,  for different values of the parameter. 

In order to prove the results for semi-algebraic sets defined by 
more general formulas than disjunctions of weak inequalities, we first
use Spanier-Whitehead duality to obtain a bound in the case of conjunctions,
and then use the construction of homotopy colimits to prove the theorem
for general ${\mathcal P}$-closed sets. Because of the use of Spanier-Whitehead
duality we get bounds on the number of stable homotopy types, rather than
homotopy types.

%
\subsection{Topology of sets defined by quadratic constraints}
%
One of the main ideas  behind our proof of Theorem \ref{the:main} is
to parametrize a construction introduced by 
Agrachev in \cite{Agrachev} while studying the topology of sets defined by 
(purely) quadratic inequalities (that is without the parameters 
$X_1,\ldots,X_k$ in  our notation). 
However, we avoid construction of Leray spectral sequences as done in 
\cite{Agrachev}. For the rest of this section, we fix 
a set of polynomials
\[
{\mathcal Q} = \{Q_1,\ldots,Q_{m}\} \subset \R[Y_0,\ldots,Y_\ell,X_1,\ldots,X_k]
\] 
which are homogeneous of degree  $2$ in $Y_0,\ldots,Y_\ell$, 
and of degree at most $d$ in $X_1,\ldots,X_k$.

We will denote by 
\[
Q = (Q_1,\ldots,Q_m): \R^{\ell+1} \times \R^k \rightarrow \R^m,
\]
the map defined by the polynomials $Q_1,\ldots,Q_m$, and
generally, for $I \subset \{1,\ldots, m\}$, we denote by 
$Q_I: \R^{\ell+1} \times \R^k \rightarrow \R^I$, the map whose
co-ordinates are given by $Q_i$, $i \in I$. 
We will often drop  the subscript~$I$ from our notation, 
when $I = [m]$.

For any subset~$I \subset [m]$, 
let
$A_I \subset \Sphere^{\ell} \times \R^k$ be the semi-algebraic set defined by
\begin{equation}
\label{eqn:defofA_I}
A_I = \bigcup_{i\in I}
\{ (\y,\x) \;\mid\; |\y|=1\; \wedge\; Q_i(\y,\x) \leq 0\},
\end{equation}
and let
\begin{equation}
\label{eqn:defofOmega_I}
\Omega_I = \{\omega \in \R^{m} \mid  |\omega| = 1, 
\omega_i = 0, i \not\in I, 
\omega_i \leq 0, i \in I\}.
\end{equation}

For $\omega \in \Omega_I$ we denote by 
${\omega}Q \in \R[Y_0,\ldots,Y_\ell,X_1,\ldots,X_k]$ the polynomial
defined by 
\begin{equation}
\label{eqn:defofomegaQ}
{\omega} Q = \sum_{i=0}^{m} \omega_i Q_i.
\end{equation}

For $(\omega,\x) \in F_I = \Omega_I \times \R^k$, we will denote by
$\omega Q(\cdot,\x)$ the quadratic form in $Y_0,\ldots,Y_\ell$ 
obtained from $\omega Q$ by specializing $X_i = \x_i, 1 \leq i \leq k$.

Let $B_I \subset \Omega_I \times \Sphere^{\ell} \times \R^k$ 
be the semi-algebraic set defined by

\begin{equation}
\label{eqn:defofB_I}
B_I = \{ (\omega,\y,\x)\mid \omega \in 
\Omega_I, 
\y\in \Sphere^{\ell}, 
\x \in\R^k,  \; {\omega}Q(\y,\x) \geq 0\}.
\end{equation}

We denote by $\phi_1: B_I \rightarrow F_I$ and 
$\phi_2: B_I \rightarrow \Sphere^{\ell} \times\R^k$ the two projection maps
(see diagram below).
\begin{equation}
\label{eqn:maindiagram}
\begin{diagram}
\node{}
\node{B_I} \arrow{sw,t}{\phi_{I,1}}\arrow{s,..}\arrow{se,t}{\phi_{I,2}} \\
\node{F_I = \Omega_I \times\R^k} \arrow{e} \node{\R^k} \node{\Sphere^{\ell} \times\R^k} \arrow{w} \\
\end{diagram}
\end{equation}
The following key proposition was proved by Agrachev \cite{Agrachev}
in the unparametrized
situation, but as we see below it works in the parametrized case as well.

\begin{proposition}
\label{prop:homotopy2}
The map $\phi_2$ gives a homotopy equivalence between $B_I$ and 
$\phi_2(B_I) = A_I$.
\end{proposition}
\begin{proof}
In order to simplify notation we prove it in the case $I = [m]$,
and the case for any other $I$ would follow immediately. 
We first prove that $\phi_2(B) = A.$
If $(\y,\x) \in A,$ 
then there exists some $i, 1 \leq i \leq m,$ such that
$Q_i(\y,\x) \leq 0$. 
Then for $\omega = (-\delta_{1,i},\ldots,-\delta_{m,i})$
(where $\delta_{i,j} = 1$ if $i=j$, and $0$ otherwise),
we see that $(\omega,\y,\x) \in B$.
Conversely,
if $(\y,\x) \in \phi_2(B),$ then there exists 
$\omega = (\omega_1,\ldots,\omega_m) \in \Omega$ such that,
\[
\sum_{i=1}^m \omega_i Q_i(\y,\x) \geq 0.
\] 
Since, $\omega_i \leq 0, 1\leq i \leq m,$ and not all $\omega_i = 0$.
This implies that $Q_i(\y,\x) \leq 0$ for
some $i, 1 \leq i \leq m$. This shows that $(\y,\x) \in A$.

For $(\y,\x) \in \phi_2(B)$, the fiber 
$$
\phi_2^{-1}(\y,\x) = \{ (\omega,\y,\x) \mid  
 \omega \in \Omega \;\mbox{such that} \;  {\omega}Q(\y,\x) \geq 0\},
$$
is a non-empty subset of $\Omega$ defined by a single linear inequality.
Thus, each non-empty fiber is an intersection of a convex cone with
$\Sphere^{m-1}$, and hence contractible.

The proposition now follows from 
the well-known Smale-Vietoris theorem \cite{Smale}. 
\end{proof}
We will use the following notation.

\begin{notation}
For any  quadratic form $Q \in \R[Y_0,\ldots,Y_\ell]$, 
we will denote by ${\rm index}(Q)$, the number of
negative eigenvalues of the symmetric matrix of the corresponding bilinear
form, that is of the matrix $M_Q$ such that,
$Q(\y) = \langle M_Q \y, \y \rangle$ for all $\y \in \R^{\ell+1}$ 
(here $\langle\cdot,\cdot\rangle$ denotes the usual inner product). 
We will also
denote by $\lambda_i(Q), 0 \leq i \leq \ell$, the eigenvalues of $Q$, in non-decreasing order, i.e.
\[ 
\lambda_0(Q) \leq \lambda_1(Q) \leq \cdots \leq \lambda_\ell(Q).
\]
\end{notation}
For $I \subset [m]$,  we denote by

\begin{equation}
\label{eqn:defofF_Ij}
F_{I,j} = \{(\omega,\x) \in \Omega_I \times \R^k \;  
\mid \;  {\rm index}({\omega}Q(\cdot,\x)) \leq j \}.
\end{equation}

It is clear that each 
$F_{I,j}$ is a closed semi-algebraic subset of 
$F_I$ and that they induce a filtration of the space
$F_I$, given by
\[
F_{I,0} \subset F_{I,1} \subset \cdots \subset F_{I,\ell+1} =F_I.
\]
\begin{lemma}
\label{lem:sphere}
The fiber of the map $\phi_{I,1}$ over a point 
$(\omega,\x)\in F_{I,j}\setminus F_{I,j-1}$ 
has the homotopy type of a sphere of dimension $\ell-j$. 
\end{lemma}

\begin{proof}
As before,  we prove the lemma
only for $I = [m]$. The proof for a general $I$ is identical.  
First notice that for
$(\omega,\x) \in  
F_{j}\setminus F_{j-1}$,
the first $j$ eigenvalues of $\omega Q(\cdot,\x)$,
\[
\lambda_0({\omega}Q(\cdot,\x)),\ldots, \lambda_{j-1}({\omega}Q(\cdot,\x)) < 0.
\] 
Moreover, letting 
$W_0({\omega}Q(\cdot,\x)),\ldots,W_{\ell}({\omega}Q(\cdot,\x))$ 
be the co-ordinates with respect to an orthonormal basis,
$e_0({\omega}Q(\cdot,\x)),\ldots,e_{\ell}({\omega}Q(\cdot,\x))$ ,
consisting of  eigenvectors of ${\omega}Q(\cdot,\x)$, we have that 
$\phi_1^{-1}(\omega,\x)$ is the subset of 
$\Sphere^{\ell} = \{\omega\} \times \Sphere^{\ell} \times \{\x\}$ 
defined by
$$
\displaylines{
\sum_{i=0}^{\ell} \lambda_i({\omega}Q(\cdot,\x))W_i({\omega}Q(\cdot,\x))^2 \geq  0, \cr
\sum_{i=0}^{\ell} W_i({\omega}Q(\cdot,\x))^2 = 1.
}
$$
Since, $\lambda_i({\omega}Q(\cdot,\x)) < 0, 0 \leq i < j,$ it follows that
for $(\omega,\x) \in F_{j}\setminus F_{j-1}$,
the fiber $\phi_1^{-1}(\omega,\x)$ is homotopy equivalent to the
$(\ell-j)$-dimensional sphere defined by setting
\[
W_0({\omega}Q(\cdot,\x)) = \cdots = W_{j-1}({\omega}Q(\cdot,\x)) = 0
\]
on the sphere defined by
$\sum_{i=0}^{\ell}W_i({\omega}Q(\cdot,\x))^2 = 1$.
\end{proof}
For each 
$(\omega,\x) \in F_{I,j} \setminus F_{I,j-1}$, let 
$L_j^+(\omega,\x) \subset \R^{\ell+1}$ denote the sum of the 
non-negative eigenspaces of 
$\omega Q(\cdot,\x)$ (i.e. $L_j^+(\omega,\x)$ is the largest linear
subspace  of $\R^{\ell+1}$ on which $\omega Q(\cdot,\x)$ is positive 
semi-definite). 
Since  ${\rm index}(\omega Q(\cdot,\x)) = j$ stays invariant as
$(\omega,\x)$ varies over $F_{I,j}\setminus F_{I,j-1}$,
$L_j^+(\omega,\x)$ varies continuously with $(\omega,\x)$.

We will denote by $C_I$ the semi-algebraic set defined by

\begin{equation}
\label{eqn:definition_of_C}
C_I = \bigcup_{j=0}^{\ell+1} \{(\omega,\y,\x) \;\mid\; (\omega,\x) \in 
      F_{I,j}\setminus F_{I,j-1}, 
\y \in L_j^+(\omega,\x), |\y| = 1\}.
\end{equation}

The following proposition 
relates the homotopy type of $B_I$ to that
of $C_I$. 
\begin{proposition}
\label{prop:homotopy1}
The semi-algebraic set~$C_I$ defined above is homotopy equivalent to $B_I$ 
(see (\ref{eqn:defofB_I}) for the definition of $B_I$).
\end{proposition}
\begin{proof}
We give a deformation retraction of $B_I$ to $C_I$ constructed as follows.
For each $(\omega,x) \in F_{I,\ell} \setminus 
F_{I,\ell-1}$, we can retract the fiber 
$\phi_1^{-1}(\omega,x)$ to the zero-dimensional sphere,
$L_{\ell}^+(\omega,x) \cap \Sphere^{\ell}$
by the following retraction. Let 
\[
W_0({\omega}Q_I(\cdot,x)),\ldots,W_{\ell}({\omega}Q_I(\cdot,x))
\] 
be the co-ordinates with respect to an orthonormal basis 
$e_0({\omega}Q(\cdot,\x)),\ldots,e_{\ell}({\omega}Q(\cdot,\x))$,
consisting of
eigenvectors of ${\omega}Q_I(\cdot,x)$ corresponding to 
non-decreasing
order of the eigenvalues of ${\omega}Q(\cdot,\x)$. Then,
$\phi_1^{-1}(\omega,x)$ is the subset of 
$\Sphere^{\ell}$  defined by
$$
\displaylines{
\sum_{i=0}^{\ell} \lambda_i({\omega}Q_I(\cdot,x))W_i({\omega}Q_I(\cdot,x))^2 \geq  0, \cr
\sum_{i=0}^{\ell} W_i({\omega}Q_I(\cdot,x))^2 = 1.
}
$$
and $L_{\ell}^+(\omega,x)$ is defined by $W_0({\omega}Q_I(\cdot,x)) = \cdots = 
W_{\ell-1}({\omega}Q_I(\cdot,x)) = 0$.
We retract 
$\phi_1^{-1}(\omega,x)$ to the zero-dimensional sphere,
$L_{\ell}^+(\omega,x) \cap \Sphere^{\ell}$
by the retraction sending,
$(w_0,\ldots,w_\ell) \in \phi_1^{-1}(\omega,x)$, at time $t$ to
$((1-t)w_0,\ldots,(1-t)w_{\ell-1},t'w_\ell)$, where $0 \leq t \leq 1$,
and 
$
\displaystyle{
t' = \left(\frac{1 - (1-t)^2 \sum_{i=0}^{\ell-1}w_i^2}{w_\ell^2}\right)^{1/2}.
}
$
Notice that even though the local co-ordinates 
$(W_0,\ldots,W_\ell)$ in $\R^{\ell+1}$ 
with respect to the orthonormal basis
$(e_0,\ldots,e_\ell)$ may not be uniquely defined at the point $(\omega,x)$ 
(for instance,
if the quadratic form ${\omega}Q_I(\cdot,x)$ has multiple eigen-values),
the retraction is still well-defined since it only depends on the
decomposition of $R^{\ell+1}$ into 
orthogonal complements $\spanof(e_0,\ldots,e_{\ell-1})$ and $\spanof(e_\ell)$ which is well defined. 
We can thus retract simultaneously all fibers
over $F_{I\ell} \setminus F_{I,\ell-1}$
continuously, to obtain a semi-algebraic set~$B_{I,\ell} \subset B_I$,
which is moreover homotopy equivalent to $B_I$.

This retraction is schematically shown in Figure \ref{fig:figure2}, where 
$F_{I,\ell}$ is the closed segment, and
$F_{I,\ell-1}$ are its end points.

\begin{figure}[hbt]
         \centerline{
           \scalebox{0.5}{
             \input{figure2.pstex_t}
             }
           }
         \caption{Schematic picture of the retraction of $B_I$ to $B_{I,\ell}$.}
         \label{fig:figure2}
       \end{figure}
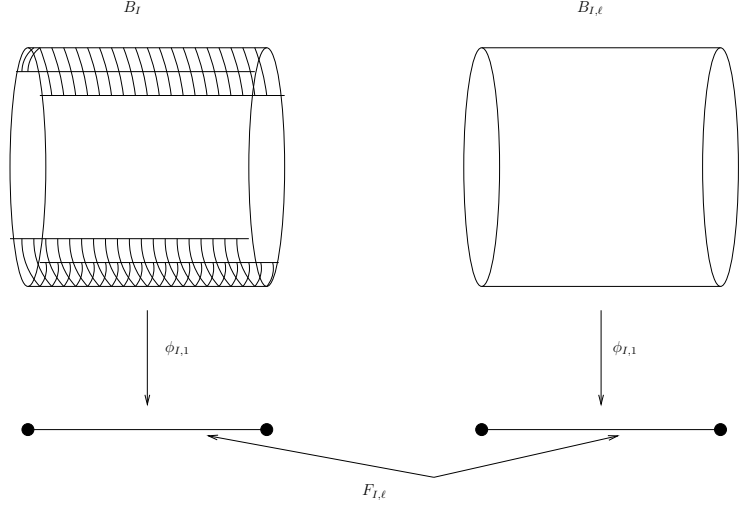

Now starting from $B_{I,\ell}$, retract all fibers over
$F_{I,\ell-1} \setminus F_{I,\ell-2}$ to
the corresponding one dimensional spheres,
by the retraction sending,
$(w_0,\ldots,w_\ell) \in \phi_1^{-1}(\omega,x)$, at time $t$ to
$((1-t)w_0,\ldots,(1-t)w_{\ell-2},t'w_{\ell-1}, t'w_\ell)$, where $0 \leq t \leq 1$,
and 
$
\displaystyle{
t' = \left(\frac{1 - (1-t)^2 \sum_{i=0}^{\ell-2}w_i^2}{\sum_{i=\ell-1}^{\ell}
w_i^2}\right)^{1/2}
}
$
to obtain $B_{I,\ell-1}$,
which is homotopy equivalent to $B_{I,\ell}$. Continuing this 
process we finally
obtain $B_{I,0} = C_I$, which is clearly homotopy equivalent to 
$B_I$ by construction.
\end{proof}

Notice that the semi-algebraic set 
$\phi_1^{-1}(F_{I,j} \setminus F_{I,j-1})\cap C_I$ is a
$\Sphere^{\ell - j}$-bundle over $F_{I,j} \setminus F_{I,j-1}$ under the
map $\phi_1$, and $C_I$ is a union of these sphere bundles. We have
good control over the bases, $F_{I,j} \setminus F_{I,j-1}$, of these bundles,
that is we have good bounds on the number as well as 
the degrees of polynomials 
used to define them. 
However, these bundles could be possibly 
glued to each other in complicated ways, and it is not immediate
how to control this glueing data, since different
types of glueing could give rise to different homotopy types of the 
underlying space. 
In order to get around this difficulty, we consider 
certain closed subsets, $F_{I,j}'$ of $F_I$, 
where each $F_{I,j}'$ is an infinitesimal
deformation of $F_{I,j} \setminus F_{I,j-1}$, 
and form the base of a $\Sphere^{\ell - j}$-bundle. Moreover, these new 
sphere bundles are glued to each other along sphere bundles over
$F_{I,j}' \cap F_{I,j-1}'$, and their union, 
$C'_I$,  is homotopy equivalent to $C_I$. Finally, the polynomials defining
the sets $F_{I,j}'$ are in general position in a very strong sense, and
this property is used later to bound the number of isotopy classes of the
sets $F_{I,j}'$ in the parametrized situation.

We now make precise the argument outlined above. Let 
$\Lambda_I$ be the polynomial in $ \R[Z_1,\ldots,Z_m,X_1,\ldots,X_k,T]$ 
defined by
\begin{eqnarray*}
\Lambda_I  &=& \det(M_{Z_I \cdot Q} + T\; {{\rm Id}}_{\ell+1}),\\
   &=&  T^{\ell+1} + H_{I,\ell} T^\ell + \cdots + H_{I,0},
\end{eqnarray*}
where $Z_I \cdot Q = \sum_{i \in I}  Z_i Q_i$, and 
each $H_{I,j} \in \R[Z_1,\ldots,Z_m,X_1,\ldots,X_k]$.

Notice, that $H_{I,j}$ is obtained from
$H_j = H_{[m],j}$ by setting
for each $i \not\in I$, the variable $Z_i$ to $0$ in the polynomial $H_j$.

Note also that for $(\z,\x) \in \R^m\times\R^k$, the polynomial
$\Lambda_I(\z,\x,T)$ being the characteristic polynomial of a real symmetric
matrix has all its roots real. 
It then follows from Descartes' rule of signs 
(see for instance \cite{BPR03}),
that for each $(\z,\x) \in \R^m \times \R^k$, 
where $\z_i = 0$ for all $ i \not\in I$,
${\rm index}(\z Q(\cdot,\x))$ is determined by the sign vector
\[
({\rm sign}(H_{I,\ell}(\z,\x)),\ldots,{\rm sign}(H_{I,0}(\z,\x))).
\]  
Hence, denoting by 
\begin{equation}
\label{eqn:defofH_I}
{\mathcal H}_I = 
\{H_{I,0},\ldots,H_{I,\ell}\} \subset \R[Z_1,\ldots,Z_m,X_1,\ldots,X_k],
\end{equation}
we have

\begin{lemma}
For each $j, 0 \leq j \leq \ell+1$, 
$F_{I,j}$ is the intersection of $F_I$
with
a ${\mathcal H_I}$-closed semi-algebraic set 
$D_{I,j} \subset \R^{m+k}$.
\end{lemma}

\begin{notation}
Let $D_{I,j}$ be defined by the formula
\begin{equation}
\label{eqn:defofD_Ij}
D_{I,j} = \bigcup_{\sigma \in \Sigma_{I,j}} \RR(\sigma),
\end{equation}
for some $\Sigma_{I,j} \subset {\rm Sign}({\mathcal H_I})$. 
Note that,
${\rm Sign}({\mathcal H}_I) \subset {\rm Sign}({\mathcal H})$ and 
$\Sigma_{I,j} \subset \Sigma_j$ for all $I \subset [m]$.

Now, let $\bar{\delta}=(\delta_\ell,\ldots,\delta_0)$ and 
$\bar{\eps}=(\eps_{\ell+1},\ldots,\eps_0)$ be infinitesimals such that 
\[
0 < \delta_0 \ll \cdots\ll \delta_{\ell} \ll \eps_{0} \ll \cdots \ll \eps_{\ell+1} \ll 1,
\]
and let 
\begin{equation}\label{eqn:defofR'}
\R' = \R\la\bar{\eps},\bar{\delta}\ra
\end{equation}

Given $\sigma \in {\rm Sign}({\mathcal H}_I)$,
and $0 \leq j \leq \ell+1$, 
we denote by $\RR(\sigma^c_j) \subset \R'^{m+k}$ 
the set defined 
by the formula $\sigma^c_j$ obtained by taking the 
conjunction of
\[
\begin{array}{l}
 -\eps_j - \delta_i \leq H_{I,i} \leq \eps_j + \delta_i \mbox{ for each } 
H_{I,i} \in {\mathcal H}_I
\mbox{ such that } \sigma(H_{I,i}) = 0, \cr
H_{I,i} \geq - \eps_j - \delta_i,  \mbox{ for each } H_{I,i} \in {\mathcal H}_I
\mbox{ such that } \sigma(H_{I,i}) = 1, \cr
H_{I,i} \leq  \eps_j + \delta_i, \mbox{ for each }  H_{I,i} \in {\mathcal H}_I 
\mbox{ such that } \sigma(H_{I,i}) = -1.
\end{array}
\]

Similarly,
we denote by
$\RR(\sigma^o_j) \subset \R'^{m+k}$ the set defined 
by the formula $\sigma^o$ obtained by taking the 
conjunction of
\[
\begin{array}{l}
- \eps_j - \delta_i < H_{I,i} <  \eps_j + \delta_i \mbox{ for each } 
H_{i,I} \in {\mathcal H}_I
\mbox{ such that } \sigma(H_{I,i}) = 0, \cr
H_{I,i} > -  \eps_j - \delta_i,  \mbox{ for each } H_{I,i} \in {\mathcal H}_I
\mbox{ such that } \sigma(H_{I,i}) = 1, \cr
H_{I,i} <  \eps_j + \delta_i, \mbox{ for each }  H_{I,i} \in {\mathcal H}_I 
\mbox{ such that } \sigma(H_{I,i}) = -1.
\end{array}
\]
\end{notation}

For each $j, 0 \leq j \leq \ell+1$, let

\begin{eqnarray}
D_{I,j}^o &=& \bigcup_{\sigma \in \Sigma_{I,j}} \RR(\sigma_j^o),\nonumber \\
D_{I,j}^c &=& \bigcup_{\sigma \in \Sigma_{I,j}} \RR(\sigma_j^c), \nonumber\\
D_{I,j}' &=& D_{I,j}^c \setminus D_{I,j-1}^o,\nonumber \\
F_{I,j}' &=& \E(F_I,\R') \cap D_{I,j}'. \label{def:Fprime}
\end{eqnarray}
where we denote by $D_{I,-1}^o = \emptyset~$. We also denote by
$F'_I = \E(F_I,\R')$.

We now note some extra properties of the sets $D'_{I,j}$'s.

\begin{lemma}
For each $j, 0 \leq j \leq \ell+1$, $D_{I,j}'$ is a 
${\mathcal H}_I'$-closed semi-algebraic set,
where 
\begin{equation}
\label{eqn:defofH_I'}
{\mathcal H}'_I = \bigcup_{i=0}^{\ell} \bigcup_{j=0}^{\ell+1}\{
H_{I,i} + \eps_j + \delta_i,
H_{I,i} - \eps_j -\delta_i\}.
\end{equation}
\end{lemma}

\begin{proof}
Follows from the definition of the sets $D_{I,j}'$.
\end{proof}
\begin{lemma}
\label{lem:local}
For 
$0 \leq j+1 < i \leq \ell+1$,
\[
D_{I,i}' \cap D_{I,j}' = \emptyset.
\]
\end{lemma}
\begin{proof}
In order to keep notation simple we prove the proposition
only for ${I = [m]}$. The proof for a general $I$ is identical. 
The inclusions,
\[
\displaylines{
D_{j-1} \subset D_j \subset D_{i-1} \subset D_i,\cr
D_{j-1}^o \subset 
D_j^c \subset D_{i-1}^o \subset D_i^c. 
}
\]
follow directly from the definitions of the sets 
\[
D_i,D_j,D_{j-1},D_i^c,D_j^c,D_{i-1}^o, D_{j-1}^o,
\]
and the fact that,
\[
\eps_{j-1} \ll \eps_{j} \ll  \eps_{i-1} \ll \eps_{i}.
\]

It follows immediately that,
\[
D_i' = D_i^c\setminus D_{i-1}^o
\]
is disjoint from $D_j^c$,
and hence from $D_j'$.
\end{proof}
We now associate to each $F'_{I,j}$  
a $(\ell - j)$-dimensional sphere bundle as follows. 
For each 
$(\omega,\x) \in F_{I,j}'' = F_{I,j}\setminus F'_{I,j-1}$, let 
$L_j^+(\omega,\x) \subset \R^{\ell+1}$ denote the sum of the 
non-negative eigenspaces of 
$\omega Q(\cdot,\x)$ (i.e. $L_j^+(\omega,\x)$ is the largest linear
subspace  of $\R^{\ell+1}$ on which $\omega Q(\cdot,\x)$ is positive 
semi-definite). 
Since  ${\rm index}(\omega Q(\cdot,\x)) = j$ stays invariant as
$(\omega,\x)$ varies over $F''_{I,j}$,
$L_j^+(\omega,\x)$ varies continuously with $(\omega,\x)$.

Let,
\[
\lambda_0(\omega,\x) \leq \cdots \leq \lambda_{j-1}(\omega,\x) < 0 \leq \lambda_j(\omega,\x) \leq \cdots \leq \lambda_{\ell}(\omega,\x),
\]
be the eigenvalues of $\omega Q(\cdot,\x)$ for 
$(\omega,\x) \in F''_{I,j}$.
There is a continuous extension of the map sending
$(\omega,\x) \mapsto L_j^+(\omega,\x)$ to 
$(\omega,\x) \in F'_{I,j}$. 

To see this observe that for $(\omega,\x) \in F''_{I,j}$ 
the block of the first $j$ (negative) eigenvalues,
${\lambda_0(\omega,\x) \leq \cdots \leq \lambda_{j-1}(\omega,\x)}$,
and hence the sum of the eigenspaces corresponding to them can be extended
continuously to any infinitesimal neighborhood of 
$F''_{I,j}$, and in particular to
$F'_{I,j}$. Now $L_j^+(\omega,\x)$ is the orthogonal
complement of the sum of the eigenspaces corresponding to the block of negative eigenvalues,
$\lambda_0(\omega,\x) \leq \cdots \leq \lambda_{j-1}(\omega,\x)$.

We will denote by 
$C'_{I,j}\subset F'_{I,j} \times \R'^{\ell+1}$ 
the semi-algebraic set defined by

\begin{equation}
\label{eqn:defofC_Ij'}
C'_{I,j}  =  \{(\omega,\y,\x) \;\mid\; (\omega,\x) \in F'_{I,j}, 
\y \in L_j^+(\omega,\x), |\y| = 1\}.
\end{equation}

Note that the projection $\pi_{I,j}: C'_{I,j} \rightarrow F'_{I,j}$,
makes $C'_{I,j}$ the total space of a $(\ell - j)$-dimensional sphere bundle
over $F'_{I,j}$.

Now observe that, 
\[
C'_{I,j-1} \cap C'_{I,j} = \pi_{I,j}^{-1}( F'_{I,j} \cap F'_{I,j-1} ),
\]
and 
\[
\pi_{I,j}|_{C'_{I,j-1} \cap C'_{I,j}}:C'_{I,j-1} \cap C'_{I,j} \rightarrow 
F'_{I,j} \cap F'_{I,j-1}
\]
is also a  $(\ell - j)$ dimensional sphere bundle over 
$F'_{I,j} \cap F'_{I,j-1}$.

Let 
\begin{equation}
\label{eqn:defofC'}
C'_I = \bigcup_{j=0}^{\ell+1} C'_{I,j}.
\end{equation}

We have that
\begin{proposition}
\label{prop:homotopy3}
$C'_I$ is homotopy equivalent to $\E(C_I,\R')$,
where $C_I$ and $\R'$ are defined in (\ref{eqn:definition_of_C}) and 
(\ref{eqn:defofR'}) respectively.
\end{proposition}
\begin{proof}
Let $\bar{\eps}=(\eps_{\ell+1},\ldots,\eps_0)$ and let 
\[
R_i=
\begin{cases}
R\la\bar{\eps},\delta_\ell,\ldots,\delta_i\ra \text{, }0\le i\le \ell,\\
R\la\eps_{\ell+1},\ldots,\eps_{i-\ell-1} \ra \text{, }\ell+1\le i\le 2\ell+2,\\
R \text{, }i = 2\ell+3.
\end{cases}
\]
First observe that $C_I = \lim_{\eps_{\ell+1}} C_I'$ where $C_I$ is the 
semi-algebraic set defined in (\ref{eqn:definition_of_C}) above. 

Now let,
\begin{eqnarray*}
C_{I,-1} &=& C_I',\\
C_{I,0} &=& \lim_{\delta_0} C_I', \\
C_{I,i} &=& \lim_{\delta_i} C_{I,i-1}, 1 \leq i \leq \ell, \\
C_{I,\ell+1} &=& \lim_{\eps_0} C_{I,\ell}, \\
C_{I,i} &=& \lim_{\eps_{i-\ell-2}} C_{I,i-1}, \ell+2\le i\le 2\ell+3.
\end{eqnarray*}

Notice that each $C_{I,i}$ is a closed and bounded semi-algebraic set.
Also, for $i\ge 0$, let $C_{I,i-1,t} \subset \R_{i}^{m+\ell+k}$
be the semi-algebraic set obtained by 
replacing $\delta_i$ (resp., $\eps_i$) in the definition of $C_{I,i-1}$
by the variable $t$.
Then, there exists $t_0 > 0$, such that for all $0 < t_1 < t_2 \leq t_0$, 
$C_{I,i-1,t_1} \subset C_{I,i-1,t_2}$. 

It follows (see Lemma 16.17 in \cite{BPR03}) that for each $i$,
$0 \leq i \leq 2\ell+3$, 
$\E(C_{I,i}, \R_i)$ is homotopy equivalent to $C_{I,i-1}$.
\end{proof}
%

\subsubsection{Partitioning the parameter space}
\label{sec:whitney}
%
The goal of this section is to prove the following proposition
(Proposition \ref{prop:main}).
The techniques used in the proof are  similar to those 
used in \cite{BV06} for proving a similar result.
We go through the
proof in detail in order to extract the right bound in terms
of the parameters $d,k,\ell$ and $m$.

\begin{proposition}
\label{prop:main}
There exists a finite set of points $T\subset\R^k$,  
with 
\[
\# T \leq (2^m\ell k d)^{O(mk)},
\]  
such that for any $\x \in \R^k$, there
exists $\z\in T$, with the following property.

There is a semi-algebraic path,
$\gamma: [0,1] \rightarrow \R'^k$ and a continuous semi-algebraic map,
$\phi: \Omega \times [0,1]  \rightarrow \Omega $ 
(see (\ref{eqn:defofOmega_I}) and (\ref{eqn:defofR'}) for the definition of 
$\Omega$ and $\R'$), 
with 
$\gamma(0) = \x$, $\gamma(1) = \z$, and
for each $I \subset [m]$,
\[
\phi(\cdot,t)|_{F'_{I,j,\x}}: F'_{I,j,\x} \rightarrow F_{I,j,\gamma(t)}'
\]
is a homeomorphism for each $0 \leq t \leq 1$. 
\end{proposition}
Before proving Proposition~\ref{prop:main} 
we need a few preliminary results. 
Let 
\begin{equation}
\label{eqn:defofH''}
{\mathcal H}'' = 
{\mathcal H}' \cup \{Z_1,\ldots,Z_m, Z_1^2 + \cdots + Z_m^2 -1\},
\end{equation}
where ${\mathcal H}' = {\mathcal H}'_{[m]}$ is defined in
(\ref{eqn:defofH_I'}) above.

Note that for each $j$, $0 \leq j \leq \ell+1$,
$F'_{I,j}$ is a ${\mathcal H}''$-closed semi-algebraic set. 
Moreover, let $\psi:\R'^{m+k}\rightarrow\R'^k$ be the projection onto the last 
$k$ co-ordinates.

\begin{notation}
\label{not:T}
We fix a finite set of points
$T \subset \R^k$ such that for every $\x \in \R^k$ 
there exists $\z \in T$ such that for every $\mathcal{H}''$-semi-algebraic set~$V$, 
the set~$\psi^{-1}(\x)\cap V$ is homeomorphic to $\psi^{-1}(\z)\cap V$. 
\end{notation}
The existence of a finite set~$T$ with this property follows from
Hardt's triviality theorem (Theorem~\ref{the:hardt}) 
and the Tarski-Seidenberg transfer principle, as well as 
the fact that the number of ${\mathcal H}''$-semi-algebraic sets is finite. 

Now, we note some extra properties of the family ${\mathcal H}''$.
\begin{lemma}\label{prop:A}
If $\sigma \in {\rm Sign}_{p}({\mathcal H}'')$, then $p \le k+m$ and
$\RR(\sigma) \subset {\R'}^{m+k}$ is a non-singular 
$(m+k-p)$-dimensional manifold
such that at every point $(\z,\x) \in \RR(\sigma)$, the
$(p \times (m+k))$-Jacobi matrix,
\[
\left( \frac{\partial P}{\partial Z_i} , \frac{\partial P}{\partial Y_j}
\right)_{P \in{\mathcal H}'',\ \sigma(P) = 0,\
1\leq i \leq m,\ 1 \leq j \leq k}
\]
has the maximal rank $p$.
\end{lemma}
\begin{proof}
Let $\E(\Sphere^{m-1},\R')$ 
be the unit sphere in $R'^m$. 
Suppose without loss of generality that
\[
\{ P \in {\mathcal H}'' |\> \sigma (P)=0 \}= \{ 
H_{i_1}-\eps_{j_1}-\delta_{i_1},
\ldots ,H_{i_{p-1}}-\eps_{j_{p-1}}-\delta_{i_{p-1}}, \sum_{i=1}^m Z_i^2-1 \}
\]
since the equation $Z_i=0$ eliminates the variable~$Z_i$ from the polynomials. 
It follows that it suffices to show that the algebraic set 
\begin{equation}\label{algset}
V=\bigcap_{r=1}^{p-1}\{ (\z, \x)\in\E(\Sphere^{m-1},\R')\times\R'^k \mid
H_{i_r}(\z, \x)=\eps_{j_r}+\delta_{i_r}
\}
\end{equation}
is a smooth $((m-1)+k-(p-1))$-dimensional manifold
such that at every point on it
the $(p \times (m+k))$-Jacobi matrix,
\[
\left( \frac{\partial P}{\partial Z_i} , \frac{\partial P}{\partial Y_j}
\right)_{P \in{\mathcal H}'',\ \sigma(P) = 0,\
1\leq i \leq m,\ 1 \leq j \leq k}
\]
has the maximal rank $p$. 

Let $p \le m+k$. Consider the semi-algebraic
map $P_{i_1,\ldots,i_{p-1}}: \Sphere^{m-1}\times\R^k \rightarrow {\R}^{p-1}$ defined by
\[
(\z,\x) \mapsto (H_{i_1}(\z,\x),\ldots,H_{i_{p-1}}(\z,\x)).
\]
By the semi-algebraic version of Sard's theorem (see \cite{BCR}), the
set of critical values of $P_{i_1,\ldots,i_{p-1}}$ is a semi-algebraic
subset~$C$ of ${\R}^{p-1}$ of dimension strictly less than $p-1$. 
Since $\bar\delta$ and $\bar\eps$ are infinitesimals, it follows that 
\[
(\eps_{j_1}+\delta_{i_1},\ldots,\eps_{j_{p-1}}+\delta_{i_{p-1}})\notin\E(C,R').
\] 
Hence, the algebraic set~$V$ 
defined in (\ref{algset}) has the desired properties, and 
the same is true for the basic semi-algebraic set~$\RR(\sigma)$.

We now prove that $p \le m+k$.
Suppose that $p > m+k$.
As we have just proved, 
\[
\{ H_{i_1}(\z, \x)=\eps_{j_1}+\delta_{i_1},\ldots ,
H_{i_{m+k-1}}(\z, \x)=\eps_{j_{m+k-1}}+\delta_{i_{m+k-1}} \}
\]
is a finite set of points. 
But the polynomial $H_{i_{p-1}}-\eps_{j_{p-1}}-\delta_{i_{p-1}}$ cannot vanish 
on each of these points as $\bar\delta$ and $\bar\eps$ are infinitesimals.
\end{proof}
\begin{lemma}\label{prop:B}
For every $\x \in \R^k$, 
and $\sigma \in {\rm Sign}_{p}({\mathcal H}''_{\x})$,
where
\[
{\mathcal H}''_{\x}= \{ P(Z_1, \ldots ,Z_m, \x)|\> P \in {\mathcal H}'' \},
\]
the following holds.
\begin{enumerate}
\item
$0 \leq p \leq m$, and $\RR(\sigma) \cap \psi^{-1}(\x)$
is a non-singular $(m-p)$-dimensional manifold
such that at every point $(\z,\x) \in \RR(\sigma) \cap \psi^{-1}(\x)$,
the {$(p \times m)$-Jacobi matrix},
\[
\left( \frac{\partial P}{\partial Z_i} \right)_{P \in{\mathcal H}_{\x}'', 
\sigma(P) = 0, 1 \leq i \leq m}
\]
has the maximal rank $p$.
\end{enumerate}
\end{lemma}
\begin{proof}
Note that $P_{\x}=P(Z_1, \ldots ,Z_m, \x)\in\R'[Z_1,\ldots,Z_m]$ 
for each 
$P\in\mathcal{H}''$ and $\x\in\R^k$. 
The proof is now identical to the proof of 
Lemma~\ref{prop:A}.
\end{proof}
\begin{lemma}\label{Whitney}
For any bounded ${\mathcal H}''$-semi-algebraic set~$V$ defined by
\[
V = \bigcup_{\sigma \in \Sigma_V \subset {\rm Sign}({\mathcal H}'')} 
\RR(\sigma),
\]
the partitions
\begin{eqnarray*}
\label{partition}
\R'^{m+k} &=& \bigcup_{ \sigma \in {\rm Sign}({\mathcal H}'')} \RR(\sigma),\\
V &=& \bigcup_{ \sigma \in \Sigma_V} \RR(\sigma),
\end{eqnarray*}
are compatible Whitney stratifications of $\R'^{m+k}$ and $V$ respectively.
\end{lemma}
\begin{proof}
Follows directly from the definition of Whitney stratification (see \cite{GM,CS}),
and Lemma~\ref{prop:A}.
\end{proof}
Fix some sign condition $\sigma \in {\rm Sign}({\mathcal H}'')$.
Recall that $(\z,\x) \in \RR (\sigma)$ is a {\em critical point}
of the map $\psi_{{\RR (\sigma)}}$ if the Jacobi matrix,
\[
\left( \frac{\partial P}{\partial Z_i} \right)_{P \in{\mathcal H}'',
\sigma(P) = 0,\
1 \leq i \leq m}
\]
at $(\z, \x)$ is not of the maximal possible rank.
The projection $\psi (\z, \x)$ of a critical point is a {\em critical value}
of $\psi_{{\RR (\sigma)}}$.

Let $C_1\subset \R'^{m+k}$
be the set of critical points of $\psi_{{\RR (\sigma)}}$
over all sign conditions
$$
\sigma \in \bigcup_{p \le m} {\rm Sign}_{p}({\mathcal H}''),
$$
(i.e., over all $\sigma \in {\rm Sign}_{p}({\mathcal H}'')$
with $\dim (\RR (\sigma)) \ge k $).
For a bounded ${\mathcal H}''$-semi-algebraic set~$V$,
let $C_1(V)\subset V$
be the set of critical points of $\psi_{{\RR (\sigma)}}$
over all sign conditions
$$\sigma \in \bigcup_{p \le m} {\rm Sign}_{p}({\mathcal H}'')\cap
\Sigma_V$$
(i.e., over all $\sigma \in \Sigma_V$ with $\dim (\RR (\sigma)) \ge k$).

Let $C_2  \subset \R'^{m+k}$
be the union of $\RR (\sigma)$ over all
$$\sigma \in \bigcup_{p > m} {\rm Sign}_{p}({\mathcal H}'')$$
(i.e., over all $\sigma \in {\rm Sign}_{p}({\mathcal H}'')$
with $\dim (\RR (\sigma)) < k$).
For a bounded ${\mathcal H}''$-semi-algebraic set~$V$,
let $C_2(V) \subset V$
be the union of $\RR (\sigma)$ over all
$$\sigma \in \bigcup_{p > m} {\rm Sign}_{p}({\mathcal H}'') \cap \Sigma_V $$
(i.e., over all $\sigma \in \Sigma_V$ with $\dim (\RR (\sigma)) < k$). 

Denote
$C  = C_1 \cup C_2$, and
$C(V)= C_1(V) \cup C_2(V)$.
\begin{lemma}\label{closed}
For each bounded ${\mathcal H}''$-semi-algebraic 
$V$,
the set~$C(V)$ is closed and bounded.
\end{lemma}
\begin{proof}
The set~$C(V)$ is bounded since $V$ is bounded.
The union $C_2(V)$ of strata of dimensions less than $k$ is closed
since $V$ is closed.

Let $\sigma_1 \in {\rm Sign}_{p_1}({\mathcal H}'') \cap \Sigma_V$,
$\sigma_2 \in {\rm Sign}_{p_2}({\mathcal H}'') \cap \Sigma_V$,
where $p_1 \le m$, $p_1 < p_2$,
and if $\sigma_1 (P)=0$, then $\sigma_2 (P)=0$ for any $P \in {\mathcal H}''$.
It follows that stratum $\RR (\sigma_2)$ lies in the closure of the stratum 
$\RR (\sigma_1)$.
Let ${\mathcal J}$ be the finite family of $(p_1 \times p_1)$-minors such that
$Z( {\mathcal J}) \cap \RR (\sigma_1)$ is the set of all critical points of
$\pi_{\RR (\sigma_1)}$.
Then $Z( {\mathcal J}) \cap \RR (\sigma_2)$ is either contained in
$C_2(V)$
(when $\dim (\RR (\sigma_2)) <k$), or is contained in the set of all critical points
of $\pi_{\RR (\sigma_2)}$ (when $\dim (\RR (\sigma_2)) \ge k$).
It follows that the closure of $Z( {\mathcal J}) \cap \RR (\sigma_1)$ lies in the union
of the following sets:
\begin{enumerate}
\item
$Z( {\mathcal J}) \cap \RR (\sigma_1)$,
\item\label{case2_a}
sets of critical points of some strata of dimensions less than $m+k- p_1$,
\item
some strata of dimension less than $k$.
\end{enumerate}
Using induction on descending dimensions in case (\ref{case2_a}), 
we conclude that the closure of
$Z( {\mathcal J}) \cap \RR (\sigma_1)$ is contained in $C(V)$.
Hence, $C(V)$ is closed.
\end{proof}
\begin{definition}
\label{def:criticalvalues}
We denote by 
$G_i = \psi(C_i), i= 1,2$, and
$G = G_1 \cup G_2$. 
Similarly, for each bounded ${\mathcal H}''$-semi-algebraic set~$V$, 
we denote by
$G_i(V) = \psi(C_i(V))$, ${i= 1,2}$, and
$G(V) = G_1(V) \cup G_2(V)$.
\end{definition}
\begin{lemma}\label{representatives}
We have 
$T \cap G = \emptyset$. 
In particular, $T \cap G(V) = \emptyset$
for every bounded {${\mathcal H}''$-semi-algebraic} set~$V$.
\end{lemma}
\begin{proof}
By Lemma~\ref{prop:B}, for all $\x\in T$, and 
$\sigma \in {\rm Sign}_{p}({\mathcal H}_{\x}'')$,
\begin{enumerate}
\item\label{lem:rep:1}
$0 \leq p \leq m$, and 
\item\label{lem:rep:2}
$\RR(\sigma) \cap \psi^{-1}(\x)$
is a non-singular $(m-p)$-dimensional manifold
such that at every point $(\z,\x) \in \RR(\sigma) \cap \psi^{-1}(\x)$,
the $(p \times m)$-Jacobi matrix,
\[
\left( \frac{\partial P}{\partial Z_i} \right)_{P \in{\mathcal H}_{\x}'', 
\sigma(P) = 0, 1 \leq i \leq m}
\]
has the maximal rank $p$.
\end{enumerate}
If a point $\x \in T \cap G_1 = T \cap \psi(C_1)$, then
there exists $\z \in \R'^m$ such that $(\z,\x)$
is a critical point of $\psi_{\RR (\sigma)}$
for some $\sigma \in \bigcup_{p \le m} {\rm Sign}_{p}({\mathcal H}'')$,
and this is impossible by (\ref{lem:rep:2}).

Similarly, $\x \in T \cap G_2 = T \cap \psi(C_2)$,
implies that there exists $\z \in \R'^m$ such that
$(\z,\x) \in \RR (\sigma)$ for some
$\sigma \in \bigcup_{p > m} {\rm Sign}_{p}({\mathcal H}'')$, and this
is impossible by (\ref{lem:rep:1}).
\end{proof}
Let $D$ be a connected component of
$\R'^{k} \setminus G$, and for a bounded
${\mathcal H}''$-semi-algebraic set~$V$,
let $D(V)$ be a 
connected component of $\psi(V) \setminus G(V)$.
\begin{lemma}
\label{lem:discriminant}
For every bounded ${\mathcal H}''$-semi-algebraic set~$V$,
all fibers $\psi^{-1}(\x) \cap V$, $\x \in D$
are homeomorphic.
\end{lemma}
\begin{proof}
Lemma~\ref{prop:B} and Lemma~\ref{Whitney} imply that
$\widehat V=\psi^{-1}(\psi(V) \setminus G(V))\cap V$ is a
Whitney stratified set having strata of dimensions at least $k$.
Moreover, $\psi|_{\widehat V}$ is a proper stratified submersion.
By Thom's first isotopy lemma (in the semi-algebraic version, over real closed fields
\cite{CS}) the map $\psi|_{\widehat V}$ is a locally trivial fibration.
In particular, all fibers $\psi^{-1}(\x)\cap V$, $\x \in D(V)$
are homeomorphic for every connected component $D(V)$.
The lemma follows, since
the inclusion $G(V) \subset G$ implies that either
$D \subset D(V)$ for some connected component $D(V)$, or
$D \cap \psi(V)= \emptyset$.
\end{proof}
\begin{lemma}
\label{components}
For each $\x \in T$,
there exists a connected component
$D$ of $\R'^k \setminus G$, such that
$\psi^{-1}(\x) \cap V$ is homeomorphic to $\psi^{-1}(\x_1) \cap V$
for every bounded ${\mathcal H}''$-semi-algebraic set~$V$ and
for every $\x_1 \in D$.
\end{lemma}
\begin{proof}
Let $V$ be a bounded ${\mathcal H}''$-semi-algebraic set and $\x \in T$. 
By Lemma~\ref{representatives}, $\x$ belongs to some connected component~$D$ 
of $\R'^k \setminus G$. Lemma~\ref{lem:discriminant} implies that 
$\psi^{-1}(\x) \cap V$ is homeomorphic to $\psi^{-1}(\x_1) \cap V$
for every $\x_1 \in D$.
\end{proof}
We now are able to proof Proposition~\ref{prop:main}.
\begin{proof}[Proof of Proposition \ref{prop:main}]
Recall that $G=G_1 \cup G_2$,
where $G_1$ is the union of sets of critical values of
$\psi_{\RR (\sigma)}$ over all strata $\RR (\sigma)$ of dimensions at least $k$,
and $G_2$ is the union of projections of all strata of dimensions less than $k$. 

By Lemma~\ref{components} it suffices to bound the number of connected components 
of the set~$\R'^k \setminus G$. 
Denote by ${\mathcal E}_1$ the family of closed sets of critical points of
$\psi_{{\mathcal Z} (\sigma)}$, over all sign conditions $\sigma$ such that
strata $\RR (\sigma)$ have dimensions at least $k$ 
(the notation ${\mathcal Z} (\sigma)$
was introduced in Section~\ref{sub:notations}).
Let ${\mathcal E}_2$ be the family of closed sets ${\mathcal Z} (\sigma)$, over all
sign conditions $\sigma$ such that strata $\RR (\sigma)$ 
have dimensions equal to $k-1$. 
Let ${\mathcal E}= {\mathcal E}_1\cup {\mathcal E}_2$. 
Denote by $E$ the image under the projection $\psi$ of the union of all sets in 
the family ${\mathcal E}$.

Because of the transversality condition,
every stratum of the stratification of $V$, having the dimension
less than $m+k$, lies in the closure of a stratum, having the next higher dimension.
In particular, this is true for strata of dimensions less than $k-1$.
It follows that $G \subset E$, and thus
every connected component of the complement 
$\R'^k \setminus E$
is contained in a connected component of 
$\R'^k \setminus G$.
Since $\dim (E)<k$, every connected component of 
$\R'^k \setminus G$
contains a connected component of 
$\R'^k \setminus E$.
Therefore, it is sufficient to estimate from above the
Betti number 
${\rm b}_0 (\R'^k \setminus E)$ 
which is equal to
${\rm b}_{k-1}(E)$ by the Alexander's duality.

The total number of sets ${\mathcal Z} (\sigma)$, such that
$\sigma \in {\rm Sign}({\mathcal H}'')$ and $\dim ({\mathcal Z} (\sigma)) \ge k-1$,
is $O(\ell^{2(m+1)})$ because each ${\mathcal Z} (\sigma)$ is
defined by a conjunction of at most $m+1$ of possible 
$O(\ell^2+m)$ 
polynomial equations.

Thus, the cardinality $\# {\mathcal E}$, as well as the
number of images under the projection $\pi$ of sets in
${\mathcal E}$ is $O(\ell^{2(m+1)})$.
According to (\ref{eq:MV1}) in Proposition~\ref{prop:MV},
${\rm b}_{k-1}(E)$
does not exceed the sum of certain Betti numbers of sets of the type
$$
\Phi =\bigcap_{1 \le i \le p} \pi (U_i),
$$
where every $U_i \in {\mathcal E}$ and $1 \leq p \leq k$.
More precisely, we have
\[
\displaylines{
{\rm b}_{k-1}(E) \;\leq \;\sum_{1 \le p \le k}\quad
\sum_{ \{ U_{1}, \ldots ,U_{p} \} \subset\ {\mathcal E}}
{\rm b}_{k-p} \left( \bigcap_{1 \le i \le p} \pi (U_i) \right).
}
\]
Obviously, there are $O(\ell^{2(m+1)k})$ 
sets of the kind $\Phi$. 

Using inequality (\ref{eq:MV2}) in
Proposition \ref{prop:MV}, we have that for each $\Phi$ as above,
the Betti number ${\rm b}_{k-p}(\Phi)$ does not exceed
the sum of certain Betti numbers of unions of the kind,
$$\Psi = \bigcup_{1 \le j \le q} \pi (U_{i_j}) =
\pi \left( \bigcup_{1 \le j \le q} U_{i_j} \right),$$
with  $1 \leq q \leq p$.
More precisely,
\begin{eqnarray*}
{\rm b}_{k-p} (\Phi) &\;\leq\;&
\sum_{1 \le q \le p}\quad \sum_{1 \leq i_1 < \cdots< i_q \leq p}
{\rm b}_{k-p+q-1} \left( \pi \left( \bigcup_{1 \le j \le q} U_{i_j} \right) \right).
\end{eqnarray*}
It is clear that there are at most $2^{p} \leq 2^k$ sets of the kind $\Psi$.

If a set~$U \in {\mathcal E}_1$, then it is defined by $m$ polynomials
of degrees at most $O(\ell d)$. 
If a set~$U \in {\mathcal E}_2$, then it is defined by $O(2^m)$ polynomials
of degrees $O(m\ell d)$,
since the critical points on strata of dimensions at least $k$
are defined by $O(2^m)$ determinantal equations, the
corresponding matrices have orders $O(m)$, and
the entries of these matrices are polynomials of 
degrees at most $O(\ell d)$.

It follows that the closed and bounded set
$$\bigcup_{1 \le j \le q} U_{i_j}$$
is defined by $O(k2^m))$ polynomials of degrees $O(\ell d)$.

By Proposition~\ref{prop:GVZ},
${\rm b}_{k-p+q-1}(\Psi) \le (2^mk\ell d)^{O(mk)}$ for all $1 \le p \le k$, $1 \le q \le p$.
Then ${\rm b}_{k-p} (\Phi) \le (2^mk\ell d)^{O(mk)}$ for every $1 \le p \le k$.
Since there are $O(\ell^{2(m+1)k})$ sets of the kind $\Phi$, we get the
claimed bound
\[
{\rm b}_{k-1}(E) \le (2^mk\ell d)^{O(mk)}.
\]
The rest of the proof follows from Proposition~\ref{components}.
\end{proof}
%
\subsection{The Homogeneous Case}
\label{subsec:homogeneous}
%
We first consider the case, where all the polynomials in
${\mathcal Q}$ are homogeneous in variables $Y_0,\ldots,Y_\ell$ and 
we bound the number of homotopy types among the fibers~$S_{\x}$, 
defined by the ${\mathcal Q}$-closed 
semi-algebraic subsets $S$ of $\Sphere^{\ell} \times \R^k$.
We first the prove the following theorems for the special cases 
of unions and intersections.
\begin{theorem}
\label{the:union}
Let $\R$ be a real closed field and let
\[
{\mathcal Q} = \{Q_1,\ldots,Q_m\} \subset \R[Y_0,\ldots,Y_\ell,X_1,\ldots,X_k],
\]
where each $Q_i$ is homogeneous of degree $2$ in the 
variables $Y_0,\ldots,Y_\ell$,
and of degree at most $d$ in $X_1,\ldots,X_k$.

For $i\in [m]$, let  
$A_i\subset \Sphere^{\ell} \times \R^k$ 
be semi-algebraic sets defined by
\[
A_i = \{ (\y,\x) \;\mid\; |\y|=1\; \wedge\; Q_i(\y,\x) \leq 0)\}, 
\]
Let $\pi: \Sphere^{\ell} \times \R^{k} \rightarrow \R^k$ be 
the projection on the last $k$ co-ordinates.

Then, the number of homotopy types amongst the fibers 
$\displaystyle{\bigcup_{i=1}^m A_{i,\x}}$
is
bounded by 
\[
(2^m\ell k d)^{O(mk)}.
\]
\end{theorem}

With the same assumptions as in Theorem \ref{the:union} we have

\begin{theorem}
\label{the:intersection}
The number of stable homotopy types amongst the fibers 
$\displaystyle{\bigcap_{i=1}^m A_{i,\x}}$ 
is
bounded by 
\[
(2^m\ell k d)^{O(mk)}.
\]
\end{theorem}

Before proving Theorems \ref{the:union} and \ref{the:intersection}
we first prove two preliminary lemmas.

\begin{lemma}
\label{lem:prelim_union}
There exists a finite set $T \subset \R^k$,
with 
\[
\# T \leq (2^m\ell kd)^{O(mk)},
\]
such that for every $\x \in \R^k$ there exists $\z \in T$, 
a semi-algebraic set 
$D_{\x,\z} \subset \R'^{m+\ell}$,
and semi-algebraic 
maps $f_\x,f_\z$, as shown in the diagram
below, such that $f_\x,f_\z$ are both homotopy equivalences.
\begin{equation}
\begin{diagram}
\node{}\node{D_{\x,\z}}\arrow{sw,tb}{f_\x}{\sim}\arrow{se,tb}{f_\z}{\sim}
\node{}\\
\node{\E(\bigcup_{i \in [m]}A_{i,\x},\R')} \node{} 
\node{\E(\bigcup_{i \in [m]}A_{i,\z},\R')}
\end{diagram}
\end{equation}

Moreover, for each $I \subset [m]$, 
there exists a subset $D_{I,\x,\z} \subset D_{\x,\z}$, such that
the restrictions, $f_{I,\x},f_{I,\z}$, 
of $f_\x,f_\z$ to $D_{I,\x,\z}$ give rise to the 
following diagram  in which all maps are again homotopy equivalences.

\begin{equation}
\begin{diagram}
\node{}\node{D_{I,\x,\z}}\arrow{sw,tb}{f_{I,\x}}{\sim}\arrow{se,tb}
{f_{I,\z}}{\sim}\node{}\\
\node{\E(\bigcup_{i \in I}A_{i,\x},\R')} \node{} \node{\E(
\bigcup_{i \in I}A_{i,\z},\R')}
\end{diagram}
\end{equation}
For each $I \subset J \subset [m]$, $D_{I,\x,\z} \subset D_{J,\x,\z}$ and
the maps $f_{I,\x},f_{I,\z}$ are restrictions of $f_{J,\x},f_{J,\z}$.
\end{lemma}
\begin{proof}[Proof of Lemma \ref{lem:prelim_union}]
By Proposition \ref{prop:main}, there exists $T \subset \R^k$ with 
\[
\#T \le (2^m\ell k d)^{O(m k)},
\]
such that for every $\x \in \R^k$, there exists
$\z \in T$, with the following property.

There is a semi-algebraic path,
$\gamma: [0,1] \rightarrow \R'^k$ and a continuous semi-algebraic map,
$\phi: \Omega \times [0,1]  \rightarrow \Omega $, 
with $\gamma(0) = \x$, $\gamma(1) = \z$, and
for each $I \subset [m]$,
\[
\phi(\cdot,t)|_{F'_{I,j,\x}}: F'_{I,j,\x} \rightarrow F'_{I,j,\gamma(t)},
\]
is a homeomorphism for each $0 \leq t \leq 1$ 
(see (\ref{eqn:defofOmega_I}), (\ref{eqn:defofR'}) 
and (\ref{def:Fprime}) for the definition of $\Omega$, $\R'$ and $F'_{I,j}$).

Now, observe that $C_{I,j,\x}'$ (resp. $C_{I,j,\z}'$) is a sphere bundle
over $F_{I,j,\x}'$ (resp. $F_{I,j,\z}'$). Moreover
\[
C'_{I,j,\x}  =  \{(\omega,\y) \;\mid\; \omega \in F'_{I,j,\x}, 
\y \in L_j^+(\omega,\x), |\y| = 1\},
\]
and, for $\omega \in F_{I,j,\x}' \cap F_{I,j-1,\x}'$, 
we have 
$L_j^+(\omega,\x) \subset L_{j-1}^+(\omega,\x)$.

We now prove that the map $\phi$ induces a homeomorphism
$\tilde{\phi}:C_{\x}' \rightarrow C_{\z}'$,
which for each $I \subset [m]$ and $0 \leq j \leq \ell$, 
restricts to a homeomorphism 
${\tilde{\phi}_{I,j}: C_{I,j,\x}' \rightarrow C_{I,j,\z}'}$. 

First recall that 
by a standard result in the theory of bundles
(see for instance, \cite{Fuks}, p.~313, Lemma~5),
the isomorphism
class of the sphere bundle 
$C_{I,j,\x}' \rightarrow F_{I,j,\x}'$, is determined by the homotopy 
class of the map,
\begin{eqnarray*}
F_{I,j,\x}' &\rightarrow & Gr(\ell+1-j,\ell+1) \\
\omega &\mapsto& L_j^+(\omega,\x),
\end{eqnarray*}
where $Gr(m,n)$ denotes the Grassmannian variety of $m$ dimensional subspaces
of $\R'^n$.

The map $\phi$ induces for each $j, 0 \leq j \leq \ell$, 
a homotopy between the maps 
\begin{eqnarray*}
f_0:F_{I,j,\x}'&\rightarrow & Gr(\ell+1-j,\ell+1)\\
\omega &\mapsto & L_j^+(\omega,\x)
\end{eqnarray*}
and 
\begin{eqnarray*}
f_1:F_{I,j,\z}'&\rightarrow & Gr(\ell+1-j,\ell+1)\\
\omega & \mapsto & L_j^+(\omega,\z)
\end{eqnarray*}
(after indentifying the sets $F_{I,j,\x}'$ and $F_{I,j,\z}'$ since they are
homeomorphic)
which respects the inclusions
$L_j^+(\omega,\x) \subset L_{j-1}^+(\omega,\x)$, and
$L_j^+(\omega,\z) \subset L_{j-1}^+(\omega,\z)$.

The above observation in conjunction with Lemma 5 in \cite{Fuks} 
is sufficient to prove the equivalence of the sphere bundles
$C_{I,j,\x}'$ and  $C_{I,j,\z}'$. 
But we need to prove a more general
equivalence, involving all the sphere bundles $C_{I,j,\x}'$ simultaneously,
for $0 \leq j \leq \ell$.
 
However, note that the proof of Lemma~5 in \cite{Fuks} proceeds by induction
on the skeleton of the CW-complex of the base of the bundle.
After choosing a sufficiently fine triangulation of the set $F_{I,\x}' \cong F_{I,\z}'$ 
compatible with the closed subsets $F_{I,j,\x}' \cong F_{I,j,\z}'$,
the same proof extends
without difficulty to this slightly more general situation
to give a fiber preserving 
homeomorphism,
$\tilde{\phi}:C_{\x}' \rightarrow C_{\z}'$,
which restricts to an isomorphism of sphere bundles,
$
\displaystyle{
\tilde{\phi}_{I,j}: C_{I,j,\x}' \rightarrow C_{I,j,\z}', 
}
$ 
for each $I \subset [m]$ and $0 \leq j \leq \ell$.

We have the following maps.

\begin{equation}
\label{eqn:diagramofmaps}
\begin{diagram}
\node{\E(A_{\x},\R')} \node{\E(B_{\x},\R')}\arrow{w,t}{\phi_2}  
\node{\E(C_{\x},\R')} \arrow{w,t}{i}  \node{C_{\x}'}\arrow{w,t}{r}
\arrow{s,r}{\tilde{\phi}} \\
\node{\E(A_{\z},\R')} \node{\E(B_{\z},\R')}\arrow{w,t}{\phi_2}  
\node{\E(C_{\z},\R')} \arrow{w,t}{i}  \node{C_{\z}'}\arrow{w,t}{r}
\end{diagram}
\end{equation}
The map $i$ is the inclusion map, and $r$ is a retraction
shown to exist by Proposition~\ref{prop:homotopy3}.

Since all the maps 
$\phi_2,i,r$
have been shown to be homotopy equivalences,
by Propositions 
\ref{prop:homotopy2}, \ref{prop:homotopy1}, and
\ref{prop:homotopy3} respectively, 
their composition is also a homotopy equivalence.

Moreover, for each $I \subset [m]$, the maps in the above diagram 
restrict properly to give a corresponding diagram:

\begin{equation}
\label{eqn:diagramofmapsI}
\begin{diagram}
\node{\E(A_{I,\x},\R')} \node{\E(B_{I,\x},\R')}\arrow{w,t}{\phi_2}  
\node{\E(C_{I,\x},\R')} \arrow{w,t}{i}  \node{C_{I,\x}'}\arrow{w,t}{r}
\arrow{s,r}{\tilde{\phi}} \\
\node{\E(A_{I,\z},\R')} \node{\E(B_{I,\z},\R')}\arrow{w,t}{\phi_2}  
\node{\E(C_{I,\z},\R')} \arrow{w,t}{i}  \node{C_{I,\z}'}\arrow{w,t}{r}
\end{diagram}
\end{equation}
Now let $D_{\x,\z} = C_{\x}'$, and $f_\x = \phi_2 \circ i\circ r$ 
and $f_\z = \phi_2 \circ i\circ r \circ \tilde{\phi}$.
Finally, for each $I \subset [m]$, let
$D_{I,\x,\z} = C_{I,\x}'$ and the maps $f_{I,\x}, f_{I,\z}$ the
restrictions of $f_\x$ and $f_\z$ respectively to
$D_{I,\x,\z}$. The collection of 
sets $D_{I,\x,\z}$ and the maps $f_{I,\x}, f_{I,\z}$ clearly satisfy
the conditions of the lemma. 
This completes the proof of the lemma.
\end{proof}

\begin{remark}
\label{rem:prelim_union}
Note that 
if $\R_1$ is a real closed sub-field of $\R$,
then Lemma \ref{lem:prelim_union} continues to hold after
we substitute ``$T \subset \R_1^k$'' and ``for all $\x \in \R_1^k$''
in place of ``$T \subset \R^k$'' and ``for all $\x \in \R^k$'' 
in the statement of the lemma. This is a consequence of the Tarski-Seidenberg
transfer principle.
\end{remark}

With the same hypothesis as in Lemma \ref{lem:prelim_union} we also have,
\begin{lemma}
\label{lem:prelim_intersection}
There exists a finite set $T \subset \R^k$,
with 
\[
\# T \leq (2^m\ell kd)^{O(mk)},
\]
such that for every $\x \in \R^k$ there exists $\z \in T$,
for each $I \subset [m]$,  
a semi-algebraic set $E_{I,\x,\z}$ 
defined over 
$\R''$, 
where $\R'' = \R\la\eps,\bar{\eps},\bar{\delta}\ra$ 
(see (\ref{eqn:defofR'} for the definition of $\bar{\eps}$ and $\bar{\delta}$), and 
S-maps
$g_{I,\x},g_{I,\z}$ 
as shown in the diagram
below such that $g_{I,\x},g_{I,\z}$ are both stable homotopy equivalences.
\begin{equation}
\begin{diagram}
\node{}\node{E_{I,\x,\z}}
\node{}\\
\node{\E(\bigcap_{i \in I}A_{i, \x},\R'')}\arrow{ne,tb}{g_\x}{\sim} 
\node{} 
\node{\E(\bigcap_{i \in I}A_{i,\z},\R'')}\arrow{nw,tb}{g_\z}{\sim}
\end{diagram}
\end{equation}

For each $I \subset J \subset [m]$, 
$E_{J,\x,\z} \subset E_{I,\x,\z}$ and 
the maps $g_{J,\x},g_{J,\z}$ are restrictions of
of $g_{I,\x},g_{I,\z}$.
\end{lemma}

\begin{proof}
Let 
$1\gg \eps > 0$ be an infinitesimal.
For $1 \leq i \leq m$,
we define
\begin{equation}
\label{eqn:tildeQ}
\tilde{Q}_i = Q_i + \eps(Y_0^2 + \cdots + Y_\ell^2),
\end{equation}
\begin{equation}
\label{eqn:tildeA} 
\tilde{A}_i = \{ (\y,\x) \;\mid\; |\y|=1\; \wedge\; \tilde{Q_i}(\y,\x)
\leq 0)\}.
\end{equation}

Note that the set 
$\displaystyle{\bigcap_{i \in I} \tilde{A}_{i,\x}}$ is homotopy equivalent to 
$\displaystyle{\E(\bigcap_{i \in I} A_{i,\x},\R\la\eps\ra)}$ 
for each $I \subset [m]$ and $\x \in \R^k$.
Applying Lemma \ref{lem:prelim_union} 
(see Remark \ref{rem:prelim_union})
to the family $\tilde{\mathcal Q} = \{-\tilde{Q}_1, \ldots, -\tilde{Q}_m\}$,
we have, 
that there exists a finite set $T \subset \R^k$,
with 
\[
\# T \leq (2^m\ell kd)^{O(mk)},
\]
such that for every $\x \in \R^k$ there exists $\z \in T$ 
such that for each $I \subset [m]$, the following
diagram
\begin{equation}
\begin{diagram}
\label{eqn:diaginlemma}
\node{}\node{\tilde{D}_{I,\x,\z}}
\arrow{sw,tb}{\tilde{f}_{I,\x}}{\sim}
\arrow{se,tb}{\tilde{f}_{I,\z}}{\sim}
\node{}\\
\node{\E(\bigcup_{i \in I}\tilde{A}_{i,\x},\R'')} \node{} 
\node{\E(\bigcup_{i \in I} \tilde{A}_{i,\z},\R'')}
\end{diagram}
\end{equation}
where for each $\x \in \R^k$ we denote 
\[
\tilde{A}_{i,\x} = \{ (\y,\x) \;\mid\; |\y|=1\; \wedge\; 
-\tilde{Q}_i(\y,\x)  \leq 0)\},
\]
$\tilde{f}_{I,\x},\tilde{f}_{I,\z}$ are homotopy equivalences.

Note that for each $\x \in \R^k$,
the set 
$
\displaystyle{
 \E(\bigcap_{i \in I} A_{i,\x},\R'')
}
$ 
is a deformation retract of the complement of 
$\displaystyle{\E(\bigcup_{i \in I}\tilde{A}_{i,\x},\R'')}$ and hence
is Spanier-Whitehead dual 
to 
$\displaystyle{\E(\bigcup_{i \in I}\tilde{A}_{i,\x},\R'')}$.
The lemma now follows by taking the Spanier-Whitehead dual of 
diagram (\ref{eqn:diaginlemma}) above for each $I \subset [m]$.
\end{proof}
\begin{proof}[Proof of Theorem \ref{the:union}]
Follows directly  from  Lemma \ref{lem:prelim_union}.
\end{proof}
\begin{proof}[Proof of Theorem \ref{the:intersection}]
Follows directly  from  Lemma \ref{lem:prelim_intersection}.
\end{proof}
We now prove a homogenous version of Theorem~\ref{the:main}

\begin{theorem}
\label{the:homogeneous}
Let $\R$ be a real closed field and let
\[
{\mathcal Q} = \{Q_1,\ldots,Q_m\} \subset \R[Y_0,\ldots,Y_\ell,X_1,\ldots,X_k],
\]
where each $Q_i$ is homogeneous of degree $2$ in the 
variables $Y_0,\ldots,Y_\ell$,
and of degree at most $d$ in $X_1,\ldots,X_k$.

Let $\pi: \Sphere^{\ell} \times \R^{k} \rightarrow \R^k$ be 
the projection on the last $k$ co-ordinates. 
Then,
for any ${\mathcal Q}$-closed semi-algebraic set~$S \subset \Sphere^\ell\times\R^k$,
the number of stable homotopy types amongst the fibers~$S_{\x}$  
is bounded by 
\[
(2^m\ell k d)^{O(mk)}.
\]
\end{theorem}

\begin{proof}
We first replace the family ${\mathcal Q}$ by the family,
\[
{\mathcal Q}' = \{Q_1,\ldots,Q_{2m}\} = 
\{Q,-Q \;\mid\; Q \in {\mathcal Q}\}.
\] 
Note that
the cardinality of ${\mathcal Q}'$ is $2m$. Let 
\[
A_i = \{ (\y,\x) \;\mid\; |\y|=1\; \wedge\; Q_i(\y,\x) \leq 0)\}.
\]
It follows from Lemma~\ref{lem:prelim_intersection} that,
there exists a set~$T\subset\R^k$ and with 
\[
\# T\le (2^m\ell k d)^{O(mk)}
\]
such that for every $I \subset [2m]$ and $\x \in \R^k$, 
there exists $\z \in T$ and 
a semi-algebraic set $E_{I,\x,\z}$ 
defined over 
$\R''=\R\la\eps,\bar{\eps},\bar{\delta}\ra$ and 
S-maps $g_{I,\x},g_{I,\z}$
as shown in the diagram
below such that $g_{I,\x},g_{I,\z}$ are both stable homotopy equivalences.
\begin{equation}
\begin{diagram}
\node{}\node{E_{I,\x,\z}}\node{}\\
\node{\E(\bigcap_{i \in I} A_{i,\x},\R'')}\arrow{ne,tb}{g_{I,\x}}{\sim} 
\node{} 
\node{\E(\bigcap_{i \in I}A_{i,\z},\R'')}
\arrow{nw,tb}{g_{I,\z}}{\sim}
\end{diagram}
\end{equation}

Now notice that each $\mathcal{Q}$-closed set~$S$ is a union of sets 
of the form $\displaystyle{\bigcap_{i \in I} A_{i}}$ with
$I \subset [2m]$. Let
\[
S = \bigcup_{I \in \Sigma \subset 2^{[2m]}} \bigcap_{i \in I} A_{i}.
\]
Moreover, the intersection of any sub-collection
of sets of the kind, $\bigcap_{i \in I} A_{i}$ with $I \subset [2m]$,  
is also a set of the same kind.
More precisely,
for any $\Sigma' \subset \Sigma$ there exists
$I_{\Sigma'} \in 2^{[2m]}$ such that
\[
\bigcap_{I \in \Sigma'} \bigcap_{i \in I} A_{i} = 
\bigcap_{i \in I_{\Sigma'}} A_{i}.
\]

We are not able to show directly a stable homotopy equivalence
between $S_\x$ and $S_\z$. Instead, we note that the 
S-maps $g_{I,\x}$ and $g_{I,\z}$ induce 
S-maps 
(cf. Definition \ref{def:hocolimit}) 
\[
\displaylines{
\tilde{g}_\x: 
\hocolimit (\{\E(\bigcap_{i \in I} A_{i,\x},\R'') \mid I \in \Sigma\}) \longrightarrow
\hocolimit (\{ E_{I,\x,\z}\mid I \in \Sigma \} ) \cr
\tilde{g}_\z: 
\hocolimit (\{\E(\bigcap_{i \in I} A_{i,\z},\R'') \mid I \in \Sigma\} ) \longrightarrow
\hocolimit (\{E_{I,\x,\z} \mid I \in \Sigma\} ) 
}
\]

which are stable homotopy equivalences by 
Lemma \ref{lem:hocolimit2}
since each $g_{I,\x}$ and $g_{I,\z}$ 
is a stable homotopy equivalence. 

Since 
$
\displaystyle{
\hocolimit (\{ \bigcap_{i \in I} A_{i,\x} \mid I \in \Sigma\})
}$ 
(resp.
$
\displaystyle{
\hocolimit (\{ \bigcap_{i \in I} A_{i,\z} \mid I \in \Sigma\})
}
$)
is homotopy equivalent by Lemma \ref{lem:hocolimit1} to 
$
\displaystyle{
\bigcup_{I \in \Sigma} \bigcap_{i \in I} A_{i,\x}
}
$
(resp. 
$
\displaystyle{
\bigcup_{I \in \Sigma} \bigcap_{i \in I} A_{i,\z}
}
$), 
it follows 
(see Remark~\ref{rem:transfer})
that 
$
\displaystyle{
S_\x = \bigcup_{I \in \Sigma}\bigcap_{i \in I} A_{i,\x}
}
$ is stable homotopy equivalent to
$
\displaystyle{
S_\z = \bigcup_{I \in \Sigma}\bigcap_{i \in I} A_{i,\z}
}
$.
This proves the theorem.
\end{proof}

\subsection{Inhomogeneous case}
%
We are now in a position to prove Theorem~\ref{the:main}.
\begin{proof}[Proof of Theorem~\ref{the:main}]
Let $\phi$ be a ${\mathcal P}$-closed formula defining the 
${\mathcal P}$-closed semi-algebraic set $S \subset \R^{\ell+k}$. 
Let $1 \gg \eps > 0$ be an infinitesimal, and let 
\[
P_0=\eps^2\left(
\sum_{i=1}^\ell Y_i^2 + \sum_{i=1}^k X_i^2
\right) - 1.
\]
Let $\tilde{\mathcal{P}}=\mathcal{P}\cup\{P_0\}$, 
and let $\tilde{\phi}$ be the 
$\tilde{\mathcal{P}}$-closed formula defined by 
\[
\tilde{\phi}=\phi\wedge \{P_0\leq 0\},
\]
defining the $\tilde{\mathcal{P}}$-closed semi-algebraic set 
$S_b\subset \R\la\eps\ra^{\ell+k}$. Note that the set~$S_b$ is bounded.

It follows from the local conical structure of semi-algebraic sets at 
infinity \cite{BCR} that the semi-algebraic set~$S_b$  
has the same homotopy type as $\E(S,\R\la\eps\ra)$. 

Considering each $P_i$ as a polynomial in the variables $Y_1,\ldots,Y_\ell$ with
coefficients in $\R[X_1,\ldots,X_k]$, and let $P_i^h$ denote the homogenization
of $P_i$. Thus, the polynomials 
$P_i^h \in \R[Y_0,\ldots,Y_\ell,X_1,\ldots,X_k]$ and are homogeneous of
degree $2$ in the variables $Y_0,\ldots,Y_\ell$. 

Let $S_b^h\subset\Sphere^\ell\times\R\la\eps\ra^{k}$
be the semi-algebraic 
set defined by the $\tilde{\mathcal{P}}^h$-closed formula 
$\tilde{\phi}^h$ (replacing $P_i$ by $P_i^h$ in $\tilde{\phi}$). 
It is clear that $S_b^h$ is a union of two disjoint,
closed and bounded semi-algebraic sets each homeomorphic to 
$S_b$, which has the same homotopy type as $\E(S,\R\la\eps\ra)$. 

The theorem is now proved
by applying Theorem~\ref{the:homogeneous} to the family 
$\tilde{\mathcal{P}}^h$ and the semi-algebraic set~$S_b^h$. 
Note that two fibers $S_{\x}$ and $S_{\y}$ are stable homotopy equivalent if and only if 
$\E(S_{\x},\R\la\eps\ra)$ and $\E(S_{\y},\R\la\eps\ra)$ 
are stable homotopy equivalent 
(see Remark \ref{rem:transfer}).
\end{proof}
%
%
\section{Metric upper bounds}
\label{sec:metric}
%
In \cite{BV06} certain metric upper bounds related to homotopy types
were proved as applications of the main result.
Similar results hold in the quadratic case, except now the bounds have
a better dependence on $\ell$. We state these results without proofs.

We first recall the following results from \cite{BV06}. 
Let $V \subset {\R}^\ell$ be a ${\mathcal P}$-semi-algebraic set, where 
${\mathcal P} \subset \Z [Y_1, \ldots , Y_\ell]$. Let
for each $P \in {\mathcal P}$,
$\deg (P) < d$, and
the maximum of the absolute values of coefficients in $P$ be less than some
constant $M$,  $0 < M \in \Z$.
For $a > 0$ we denote by $B_\ell(0,a)$
the open ball of radius  $a$ in $\R^\ell$ centered at the origin.
\begin{theorem}\label{the:ball}
There exists a constant $c > 0$, such that for any
$r_1 > r_2 > M^{d^{c\ell}}$ we have,
\begin{enumerate}
\item
$V \cap B_\ell(0,r_1)$ and $V \cap B_\ell(0,r_2)$ are homotopy equivalent, and
\item
$V \setminus B_\ell(0,r_1)$ and $V \setminus B_\ell(0,r_2)$ are homotopy equivalent.
\end{enumerate}
\end{theorem}
In the special case of quadratic polynomials we get the following improvement 
of Theorem~\ref{the:ball}.
\begin{theorem}\label{the:ballquad}
Let $\R$ be a real closed field. 
Let $V \subset {\R}^\ell$ be a ${\mathcal P}$-semi-algebraic set, where 
\[
{\mathcal P} = \{P_1,\ldots,P_m\} \subset \R[Y_1,\ldots,Y_\ell],
\]
with
${\rm deg}(P_i) \leq 2$, $1 \leq i \leq m$ and
the maximum of the absolute values of coefficients in ${\mathcal P}$ is less than some
constant $M$,  $0 < M \in \Z$.

There exists a constant $c > 0$, such that for any
$r_1 > r_2 > M^{\ell^{cm}}$ we have,
\begin{enumerate}
\item\label{the:ballquad:1}
$V \cap B_\ell(0,r_1)$ and $V \cap B_\ell(0,r_2)$ are stable homotopy equivalent, and
\item\label{the:ballquad:2}
$V \setminus B_\ell(0,r_1)$ and $V \setminus B_\ell(0,r_2)$ are stable homotopy equivalent.
\end{enumerate}
\end{theorem}
%

\end{document}

%% file: figure2.pstex_t
\begin{picture}(0,0)%
\includegraphics{figure2.pstex}%
\end{picture}%
\setlength{\unitlength}{3947sp}%
\begingroup\makeatletter\ifx\SetFigFont\undefined%
\gdef\SetFigFont#1#2#3#4#5{%
  \reset@font\fontsize{#1}{#2pt}%
  \fontfamily{#3}\fontseries{#4}\fontshape{#5}%
  \selectfont}%
\fi\endgroup%
\begin{picture}(9170,6289)(664,-7244)
\put(8251,-5386){\makebox(0,0)[lb]{\smash{{\SetFigFont{12}{14.4}{\familydefault}{\mddefault}{\updefault}{\color[rgb]{0,0,0}$\phi_{I,1}$}%
}}}}
\put(2101,-1111){\makebox(0,0)[lb]{\smash{{\SetFigFont{12}{14.4}{\familydefault}{\mddefault}{\updefault}{\color[rgb]{0,0,0}$B_I$}%
}}}}
\put(7801,-1111){\makebox(0,0)[lb]{\smash{{\SetFigFont{12}{14.4}{\familydefault}{\mddefault}{\updefault}{\color[rgb]{0,0,0}$B_{I,\ell}$}%
}}}}
\put(5101,-7186){\makebox(0,0)[lb]{\smash{{\SetFigFont{12}{14.4}{\familydefault}{\mddefault}{\updefault}{\color[rgb]{0,0,0}$F_{I,\ell}$}%
}}}}
\put(2626,-5386){\makebox(0,0)[lb]{\smash{{\SetFigFont{12}{14.4}{\familydefault}{\mddefault}{\updefault}{\color[rgb]{0,0,0}$\phi_{I,1}$}%
}}}}
\end{picture}%